\documentclass[a4paper,twoside]{amsart}
\usepackage[utf8]{inputenc}
\usepackage[british]{babel}
\usepackage{amssymb,amsmath,amsthm,amscd}
\usepackage{mathrsfs}
\usepackage{enumerate}
\usepackage{verbatim}
\usepackage{lmodern}

\newcommand{\C}{\ensuremath{\mathbb{C}}}
\newcommand{\R}{\ensuremath{\mathbb{R}}}
\newcommand{\N}{\ensuremath{\mathbb{N}}}

\newcommand{\sphere}{\ensuremath{\mathbb{S}}}

\renewcommand\epsilon\varepsilon

\theoremstyle{definition}
\newtheorem{thmA}{Theorem}

\newtheorem{corB}[thmA]{Corollary}
\newtheorem{thm}{Theorem}[section]
\newtheorem{dfn}[thm]{Definition}
\newtheorem{lem}[thm]{Lemma}
\newtheorem{prp}[thm]{Proposition}
\newtheorem{cor}[thm]{Corollary}

\newtheorem{rem}[thm]{Remark}

\author{Tim de Laat}
\thanks{TdL is a Postdoctoral Fellow of the Research Foundation -- Flanders (FWO), MM is partially supported by the Grant-in-Aid for Young Scientists (B), no.~25800033 from the JSPS, and MdlS is partially
supported by ANR grants GAMME, NEUMANN and OSQPI}
\address{Tim de Laat, 
\newline KU Leuven, Department of Mathematics
\newline Celestijnenlaan 200B -- Box 2400, B-3001 Leuven, Belgium}
\email{tim.delaat@wis.kuleuven.be}

\author{Masato Mimura}
\address{Masato Mimura
\newline Mathematical Institute, Tohoku University
\newline 980-8578, Sendai, Japan}
\email{mimura-mas@m.tohoku.ac.jp}

\author{Mikael de la Salle}
\address{Mikael de la Salle
\newline CNRS-ENS de Lyon, UMPA UMR 5669
\newline F-69364 Lyon cedex 7, France}
\email{mikael.de.la.salle@ens-lyon.fr}

\title[Strong property (T) and fixed point properties]{On strong property (T) and fixed point properties for Lie groups}

\begin{document}

\begin{abstract}
We consider certain strengthenings of property (T) relative to Banach spaces that are satisfied by high rank Lie groups. Let $X$ be a Banach space for which, for all $k$, the Banach--Mazur distance to a Hilbert space of all $k$-dimensional subspaces is bounded above by a power of $k$ strictly less than one half. We prove that every connected simple Lie group of sufficiently large real rank depending on $X$ has strong property (T) of Lafforgue with respect to $X$. As a consequence, we obtain that every continuous affine isometric action of such a high rank group (or a lattice in such a group) on $X$ has a fixed point. This result corroborates a conjecture of Bader, Furman, Gelander and Monod. For the special linear Lie groups, we also present a more direct approach to fixed point properties, or, more precisely, to the boundedness of quasi-cocycles. Without appealing to strong property (T), we prove that given a Banach space $X$ as above, every special linear group of sufficiently large rank satisfies the following property: every quasi-$1$-cocycle with values in an isometric representation on $X$ is bounded.
\end{abstract}

\maketitle

\section{Introduction and statement of the main results}\label{sec=introduction}
A locally compact group has property (T) if its trivial representation is isolated in the unitary dual of the group equipped with the Fell topology. This property was introduced in 1967 by Kazhdan in order to show that certain groups are finitely generated \cite{kazhdan}. Since then, property (T) has been a key ingredient in several striking results in different areas of mathematics. It is well known that connected higher rank simple Lie groups, i.e.~connected simple Lie groups with real rank at least $2$, and their lattices satisfy property (T).

This article deals with three strengthenings of property (T) relative to Banach spaces, namely Lafforgue's strong property (T), property (F$_X$) (for a Banach space $X$) of Bader, Furman, Gelander and Monod and property (FF$_X$), which was defined by the second-named author. We mainly consider the question whether high rank simple Lie groups satisfy these strengthenings with respect to certain natural classes of Banach spaces. In this article, we will only work with real Lie groups and real Banach spaces. A posteriori, all results also hold for complex Banach spaces by considering them as real ones.

The first version of property (T) with respect to a Banach space $X$, denoted by property (T$_X$), was given by Bader, Furman, Gelander and Monod in terms of representations having almost invariant vectors \cite{baderfurmangelandermonod} (see also Definition~\ref{def=ffpp}). The definition of Lafforgue's strong property (T) (see \cite{lafforguestrongt}, \cite{lafforguefastfourier}) is based on the characterization of property (T) in terms of the existence of a Kazhdan projection in the universal $C^{\ast}$-algebra $C^{\ast}(G)$ of $G$, i.e.~a self-adjoint idempotent $P$ such that for every unitary representation $\pi$ of $G$, the operator $\pi(P)$ is a projection onto the subspace of $\pi(G)$-invariant vectors. If we consider isometric representations on a Banach space $X$ instead of unitary representations, we obtain a version of property (T) with respect to $X$, which we denote by (T$^{\mathrm{proj}}_X$), where the superscript ``proj'' stands for ``projection''. Allowing the representations on $X$ to have small exponentional growth gives the definition of Lafforgue's strong property (T) with respect to $X$. We refer to Section \ref{sec=strongt} for details. Strong property (T) originated from Lafforgue's work on the Baum--Connes Conjecture.

In general, property (T$^{\mathrm{proj}}_X$) and property (T$^{\mathrm{strong}}_X$) are strictly stronger than property (T$_X$). To see this, note that if $\pi$ is a representation of $G$ on a Banach space $X$ such that the representation $\overline \pi \colon G \to B(X/X^{\pi(G)})$ has almost invariant vectors, then for every measure $m$ on $G$ with $\int 1 dm=1$, we have $\| \overline \pi(m)\|_{B(X/X^{\pi(G)})} \geq 1$. In particular, $\pi(m)$ is at distance at least $1$ from any projection onto the space $X^{\pi(G)}$ of invariant vectors. However, for a superreflexive Banach space $X$, property (T$_X$) is equivalent to a ``non-uniform version'' of property (T$^{\mathrm{proj}}_X$)  \cite{delaatdelasalle2} (see also \cite{drutunowak}).

Our first result states that for a large class $\mathcal{E}$ of Banach spaces, every connected simple Lie group of sufficiently large real rank has strong property (T) with respect to the Banach spaces in $\mathcal{E}$. Let us make this statement precise. For a Banach space $X$, we consider the sequence
\[
  d_k(X) = \sup\{d(E,\ell^2_{\dim E}) \mid E \subset X, \dim E \leq k\},
\]
where $d$ denotes the Banach--Mazur distance (see Section \ref{subsec=bs}). This sequence gives quantitative information on the geometry of the Banach space $X$ and describes, in a way, how similar $X$ is to a Hilbert space. It is classical \cite{john} that for every $X$ and every $k \geq 1$, we have $d_k(X) \leq k^{\frac 1 2}$. Let $\beta < \frac{1}{2}$. In what follows, we consider Banach spaces $X$ for which
\begin{equation} \label{eq=d_kX_grows_slowly}
  \exists C>0 \textrm{ such that }d_k(X) \leq C k^\beta\textrm{ for all }k\geq 1.
\end{equation}
\begin{thmA} \label{thm=strongt}
For every $\beta < \frac{1}{2}$, there exists an integer $N \geq 2$ such that every connected simple Lie group $G$ of real rank at least $N$ has strong property (T) with respect to the Banach spaces satisfying \eqref{eq=d_kX_grows_slowly}.
\end{thmA}
\begin{rem}\label{rem=remstrongt}\mbox{}
\begin{enumerate}[(i)]
\item By \cite[Proposition 4.3]{lafforguestrongt} (see also \cite[Proposition 5.2]{lafforguefastfourier}), it is known that strong property (T) passes to cocompact lattices. Hence, it is immediate that Theorem \ref{thm=strongt} also holds for cocompact lattices in $G$. It is not known whether strong property (T) passes to non-cocompact lattices.
\item It is known from the work of Lafforgue \cite{lafforguestrongt} and the work of Liao \cite{liao} that for a non-Archimedean local field $F$, any connected almost $F$-simple algebraic group with $F$-split rank at least $2$ has strong property (T) with respect to all Banach spaces with type $>1$ (the notion of type is recalled in Section \ref{subsec=bs}). One of the numerous characterizations of the fact that a Banach space $X$ has type $>1$ is that $\lim_{k \to \infty} k^{-\frac 1 2} d_k(X)= 0$, and it is an open problem whether all Banach spaces of type $>1$ satisfy \eqref{eq=d_kX_grows_slowly} for some $\beta < \frac{1}{2}$ (see \cite[Problem 27.6]{tomczakbook}). It is natural to expect that as in the non-Archimedean case, every connected simple Lie group $G$ of real rank at least $2$ has strong property (T) with respect to the Banach spaces of type $>1$. This is still open. The first steps towards such a result were provided in \cite{delasalle1} and \cite{delaatdelasalle1}, the main results of which imply that every connected simple Lie group with real rank at least $2$ has strong property (T) with respect to the Banach spaces satisfying \eqref{eq=d_kX_grows_slowly} for $\beta<\frac{1}{10}$. Theorem \ref{thm=strongt} is another step towards a real analogue of Lafforgue's and Liao's results.
\item In \cite[Theorem 5.8]{delaatdelasalle2}, the first-named author and the third-named author showed that for every $\beta< \frac 1 2$, there exists an $N\geq 2$ such that every connected simple Lie group $G$ of real rank at least $N$ has property (T$^{\mathrm{proj}}_X$) with respect to the superreflexive Banach spaces satisfying \eqref{eq=d_kX_grows_slowly}. The condition that the space is superreflexive was used through a result of Shalom asserting that isometric representations of semisimple Lie groups on superreflexive Banach spaces have the Howe-Moore property (see \cite[Theorem 9.1]{baderfurmangelandermonod}). In fact, a much older result of Veech \cite{veech} asserts that the same conclusion holds more generally for reflexive spaces. This implies that \cite[Theorem 5.8]{delaatdelasalle2} (and hence also \cite[Theorem 1.4]{delaatdelasalle2}) also holds more generally for reflexive spaces. Theorem \ref{thm=strongt} provides a strengthening of the result in \cite{delaatdelasalle2}. Indeed, we prove strong property (T) rather than property (T), and the (super)reflexivity assumption is not needed. However, in \cite{delaatdelasalle2}, the condition on $N$ is $N > \max\{8,\frac{2}{1-2\beta} - 1\}$, whereas here we have the stronger condition $N > \max\{8,\frac{3}{1-2\beta} - 3\}$. We refer to the beginning of Section \ref{sec=strongt} for further comparison between Theorem \ref{thm=strongt} and the work in \cite{delaatdelasalle2}.
\item\label{item:interpolation} As in \cite{delasalle1} and \cite{delaatdelasalle1}, the conclusion of Theorem \ref{thm=strongt} also holds with \linebreak $N>\max(8,\frac{3}{1-2\beta} - 3)$ if $X$ is a complex interpolation space between a Banach space satisfying \eqref{eq=d_kX_grows_slowly} and an arbitrary Banach space. It is unknown whether there exists a  Banach space satisfying \eqref{eq=d_kX_grows_slowly} for some $\beta<\frac 1 2$ (or more generally a space of type $>1$) which is not a complex interpolation space between a space satisfying \eqref{eq=d_kX_grows_slowly} for $\beta = 10^{-10}$ and an arbitrary Banach space.
\end{enumerate}
\end{rem}
It was proved by Lafforgue that if $G$ has strong property (T) with respect to $X \oplus \C$, then every continuous affine isometric action of $G$ on $X$ has a fixed point, i.e.~$G$ has property (F$_X$) in the terminology of \cite{baderfurmangelandermonod}. Hence, as a consequence of Theorem \ref{thm=strongt}, we obtain the following result.
\begin{thmA} \label{thm=fpp}
For every $\beta < \frac{1}{2}$, there exists an integer $N \geq 2$ such that every continuous affine isometric action of a connected simple Lie group $G$ of real rank at least $N$ on a Banach space $X$ satisfying  \eqref{eq=d_kX_grows_slowly} has a fixed point, i.e. the group $G$ has property (F$_X$).
\end{thmA}
Although it is not known whether strong property (T) passes to non-cocompact lattices, it is known from \cite{baderfurmangelandermonod} that under certain conditions property (F$_X$) with respect to a class of Banach spaces passes from a locally compact group to its lattices. These conditions are satisfied, and we obtain the following corollary.
\begin{corB}\label{cor=fpplattices}
For every $\beta < \frac{1}{2}$, there exists an integer $N \geq 2$ such that every lattice in a connected simple Lie group of real rank at least $N$ has property (F$_X$) for every Banach space $X$ satisfying  \eqref{eq=d_kX_grows_slowly}.
\end{corB}
\begin{rem} \label{rem=fx}\mbox{}
\begin{enumerate}[(i)]
\item Theorem \ref{thm=fpp} (as well as the result of the first-named author and the third-named author mentioned above) corroborates a conjecture of Bader, Furman, Gelander and Monod, asserting that connected semisimple Lie groups with higher rank simple factors and finite center and lattices in such groups have property (F$_X$) for every superreflexive Banach space $X$ (see \cite[Conjecture 1.6]{baderfurmangelandermonod} for a more precise and slightly stronger formulation of this conjecture).
\item Contrary to the case of Hilbert spaces, in which property (T) is equivalent to property (FH), in general property (F$_X$) is not implied by property (T$_X$) or property (T$^{\mathrm{proj}}_X$). Theorem \ref{thm=fpp} was therefore not a formal consequence of \cite{delaatdelasalle2}, not even for superreflexive spaces.
\item A formal consequence of Theorem \ref{thm=fpp} is that for $G$ and $X$ as in the theorem, the group $G$ has property ($\overline{\mathrm{F}}_X$): every uniformly equicontinuous affine action of $G$ on $X$ has a fixed point. Indeed, for a uniformly equicontinuous affine action of $G$ on $X$, there is an equivalent norm on $X$ for which the action is affine and isometric, and replacing the norm on $X$ by an equivalent norm preserves condition \eqref{eq=d_kX_grows_slowly} without changing $\beta$.
\end{enumerate}
\end{rem}
A second application of strong property (T) is as follows: if a locally compact group $G$ has strong property (T) (actually property (T$^{\mathrm{proj}}$) is enough) with respect to a class $\mathcal{E}$ of Banach spaces that is stable under taking vector-valued $L^2$-spaces, then the expanders coming from a lattice in $G$ do not coarsely embed into any Banach space in $\mathcal{E}$. This was proved by Lafforgue in \cite{lafforguestrongt} (see also \cite{lafforguefastfourier}). However, it turns out that for this result on the non-coarse-embeddability of expanders, one does not need the full strength of strong property (T) or property (T$^{\mathrm{proj}}$). In \cite{delaatdelasalle2}, the first-named author and the third-named author observed that in fact a form of Banach property (T) for a restricted family of representations (namely certain representations on vector-valued $L^2$-spaces) suffices to prove this result. This allowed them to prove that if $X$ satisfies \eqref{eq=d_kX_grows_slowly}, then the expanders coming from connected simple Lie groups with sufficiently high real rank do not coarsely embed into $X$, even though they could not prove property (T$^{\mathrm{proj}}_X$) when $X$ is not reflexive.

Recently, Oppenheim showed that certain groups that are not realized as lattices in Lie groups, e.g.~certain groups acting on simplicial complexes and certain (Kac--Moody--)Steinberg groups, satisfy a Banach space strengthening of property (T) that he calls \emph{robust Banach property (T)} \cite{oppenheim}. This property is slightly weaker than strong property (T), but still has the same consequences (on property (F$_X$) and the non-coarse-embeddability of expanders) as strong property (T).

A second direction that we investigate in this article is a more direct approach to fixed point properties, or, more precisely, to boundedness properties of \linebreak (quasi-)$1$-cocycles. We are able to use this approach in the setting of special linear Lie groups. A locally compact group $G$ is said to have property (FF$_X$) if for every isometric representation $\rho \colon G \to O(X)$, any quasi-$1$-cocycle $c\colon G\to X$ into $\rho$ is bounded (see Section \ref{subsec=ffx} for definitions and details). Our main result on property (FF$_X$) is as follows.
\begin{thmA} \label{thm=ffpp}
For every $\beta < \frac{1}{2}$, there exists an integer $N \geq 3$ such that $\mathrm{SL}(N,\mathbb{R})$ has property (FF$_X$) for all $X$ satisfying \eqref{eq=d_kX_grows_slowly}. 
\end{thmA}
\begin{rem} \label{rem=fpffp}\mbox{}
\begin{enumerate}[(i)]
\item Property (FF$_X$) is a boundedness property for continuous rough actions, i.e.~actions up to a uniformly bounded error, by affine isometries on $X$. It was introduced by the second-named author in \cite{mimura} as a Banach space version of property (TT) of Monod (see \cite{monod}).
\item For general $X$, it is not the case that property (FF$_X$) implies property (F$_X$). For example, let $X=\ell^1_0$ denote the zero-sum subspace of $\ell^1$ over a countable set. Now \cite[Example~2.23]{baderfurmangelandermonod} shows that any (infinite) countable discrete group fails to have property (F$_{\ell^1_0}$), whereas we will see in Remark~\ref{rem=ell1} that $\mathrm{SL}(4,\mathbb{Z})$ has property (FF$_{\ell^1_0}$). 

However, if $X$ satisfies the property that every continuous group action on $X$ by affine isometries with bounded orbits has a global fixed point, then property (FF$_X$) implies property (F$_X$). Note that every reflexive Banach space satisfies this property, as follows from the Ryll-Nardzewski fixed point theorem. Hence, Theorem \ref{thm=ffpp} establishes property (F$_X$) for high rank special linear Lie groups with respect to large classes of (in particular reflexive) Banach spaces. This approach is more direct than the approach through strong property (T).
\item In comparison with \cite[Theorem 1.4]{monodshalom}, Theorem \ref{thm=ffpp} and Corollary \ref{cor=ffppforSLnZ} are of special interest if the corresponding isometric representation is not coming from the contragredient representation of an isometric representation on a separable Banach space.
\end{enumerate}
\end{rem}
In the proof of Theorem \ref{thm=ffpp}, we will deduce property (FF$_X$) from the aforementioned property (T$^{\mathrm{proj}}_X$). In the study of property (F$_X$) in \cite{baderfurmangelandermonod}, a version of the Howe--Moore property in the setting of superreflexive Banach spaces due to Shalom is used (see Appendix 9 in \cite{baderfurmangelandermonod}). As mentioned before, the Howe--Moore property holds more generally in the setting of reflexive Banach spaces, as was proved by Veech \cite{veech}, but we do not see how to extend the arguments to Banach spaces for which the Howe-Moore property does not hold. In this article, we exploit a different method, based on previous work of the second-named author \cite{mimura}.

This article is organized as follows. We recall some preliminaries in Section \ref{sec=preliminaries}. In \ref{sec=properties}, we give precise definitions of strong property (T), property (F$_X$) and property (FF$_X$), we provide certain relevant background information, and we explain how Corollary \ref{cor=fpplattices} follows from Theorem \ref{thm=fpp}. We prove Theorem \ref{thm=strongt} in Section \ref{sec=strongt}. Theorem \ref{thm=ffpp} is proved in Section \ref{sec=ffpp}.

\section{Preliminaries} \label{sec=preliminaries}
\subsection{Polar decomposition of Lie groups} \label{subsec=kak}
Let $G$ be a connected (semi)simple Lie group with Lie algebra $\mathfrak{g}$. A polar/KAK decomposition of $G$ is given by $G=KAK$, where $K$ is such that its Lie algebra $\mathfrak{k}$ comes from a Cartan decomposition $\mathfrak{g}=\mathfrak{k} + \mathfrak{p}$ and $A$ is an abelian Lie group such that its Lie algebra $\mathfrak{a}$ is a maximal abelian subspace of $\mathfrak{p}$. The real rank of $G$ is defined as the dimension of $\mathfrak{a}$. In general, given a polar decomposition $g=k_1ak_2$ of an element $g$, where $k_1,k_2 \in K$ and $a \in A$, the element $a$ is not uniquely determined. However, after choosing a set of positive roots and restricting to the closure $\overline{A^{+}}$ of the positive Weyl chamber $A^{+}$, we still have the decomposition $G=K\overline{A^{+}}K$, but now, the element $a \in \overline{A^{+}}$ in the decomposition $g=k_1ak_2$ is uniquely determined. For details on this decomposition, we refer to \cite[Section IX.1]{helgasonlie}.

\subsection{Geometry of Banach spaces} \label{subsec=bs}
Two Banach spaces $X$ and $Y$ are said to be $C$-isomorphic if there exists an isomorphism $u\colon X \rightarrow Y$ such that $\|u\| \|u^{-1}\|\leq C$. The infimum of such constants $C$ is called the Banach--Mazur distance between $X$ and $Y$ and is denoted by $d(X,Y)$. It is known that if $X$ has dimension $k$, then $d(X,\ell^2_k) \leq k^{\frac{1}{2}}$. We have equality for $X=\ell^1_k$, i.e.~$d(\ell^1_k,\ell^2_k) = k^{\frac{1}{2}}$ for all $k \geq 1$.

Let $(g_i)_{i \in \N}$ be a sequence of independent complex Gaussian $\mathcal N(0,1)$ random variables defined on some probability space $(\Omega,\mathbb P)$. A Banach space $X$ is said to have type $p \geq 1$ if there exists a constant $T$ such that for all $n \in \N$ and all $x_1,\dots,x_n \in X$, we have $\| \sum_i g_i x_i \|_{L^2(\Omega;X)} \leq T \left(\sum_i \|x_i\|^p\right)^{1/p}$. The best $T$ is denoted by $T_p(X)$. A Banach space $X$ is said to have cotype $q \leq \infty$ if there exists a constant $C$ such that for all $n \in \N$ and all $x_1,\dots,x_n \in X$, we have $\left(\sum_i \|x_i\|^q\right)^{1/q} \leq C \| \sum_i g_i x_i\|_{L^2(\Omega;X)}$. The best $C$ is denoted by $C_q(X)$. 

Hilbert spaces have type $2$ and cotype $2$. It was proved by Kwapie{\'n} that this property characterizes the Banach spaces that are isomorphic to a Hilbert space. Superreflexive spaces have nontrivial type. On the other hand, there are spaces of nontrivial type that are not even reflexive. For every $q>2$, there are Banach spaces that are not reflexive but have type $2$ and cotype $q$ \cite{pisierxu}.

For details on the Banach--Mazur distance, type and cotype, we refer to \cite{tomczakbook}.

\subsection{Representations}
In this article, we consider linear representations of locally compact groups on Banach spaces that are strongly continuous, i.e.~the map $G \to X$ given by $g \mapsto \pi(g)x$ is continuous for every $x \in X$. Whenever a representation occurs, it is always assumed to be linear and strongly continuous, unless explicitly stated otherwise.

For a Banach space $X$, we denote by $O(X)$ the group of invertible linear isometries from $X$ to $X$. An isometric representation of a locally compact group $G$ is a (strongly continuous linear) representation $\pi \colon G \to O(X)$.

If $m$ is a compactly supported signed Borel measure on $G$ and $\pi:G \rightarrow B(X)$ is a representation, we denote by $\pi(m)$ the operator defined by $\pi(m) \xi = \int \pi(g)\xi dm(g)$ (Bochner integral) for all $\xi \in X$.

The contragredient representation ${}^t\pi$ of a representation $\pi$ of $G$ on $X$ is the representation of $G$ on $X^*$ given by $g \mapsto \pi(g^{-1})^*$. It might not be strongly continuous, even not if $\pi$ is.

\subsection{From estimates on invariant coefficients to estimates on finite type coefficients}
In our proof of Theorem \ref{thm=strongt}, we will use the following result, which was proved in \cite[Proposition 2.8]{delasalle1} under the assumption that $U$ is abelian. The setting is the following. Let $K$ be a compact Lie group with a left-invariant Riemannian metric $d$, let $U \subset K \times K$ be a closed subgroup. Let $\lambda$ be the left regular representation of $K$ on $L^2(K)$. We denote by $u \cdot k$ the action of an element $u = (k',k'') \in U$ on an element $k \in K$ by left-right multiplication. For $k \in K$ we denote by $U_k \subset U$ the stabilizer of $k$ for the action of $U$ on $K$. If $\rho$ is a finite-dimensional unitary representation of $U$ on a Hilbert space $V$, we denote by $V_k \subset V$ the space of $\rho(U_k)$-invariant vectors in $V$, and by $P_k \in B(V)$ the orthogonal projection onto $V_k$.

\begin{prp} \label{prop:nonequivariant}
For every $k_0 \in K$ and every finite-dimensional unitary representation $\rho$ of $U$ on $V$, there exists a constant $C_{V,k_0} > 0$ such that the following holds: for every Banach space $X$, every isometric representation $\pi \colon K \to O(X)$ and every map $f \colon B(X) \to V$ of the form $f(a) = \sum_{i=1}^n \langle a \xi_i,\eta_i\rangle v_i$ satisfying
\begin{itemize}
 \item $\xi_i \in X$, $\eta_i \in X^{\ast}$ and $v_i \in V$,
 \item $\sum_{i=1}^n \|\xi_i\|_X \|\eta_i\|_{X^*} \|v_i\|_V \leq 1$,
 \item $f(\pi(k') a \pi(k''^{-1})) = \rho(k',k'') f(a)$ for all $a \in B(X)$ and $(k',k'') \in U$,
\end{itemize}
we have
\begin{multline*}
  \|f(\pi(k)) - f(\pi(k_0))\|_V \leq \\ C_{V,k_0}\left(d(k,k_0) + \left\| \left(\int_{U} (\lambda(u \cdot k) - \lambda(u \cdot k_0))du\right) \otimes \mathrm{Id}_X \right\|_{B(L^2(K;X))} \right)
\end{multline*}
for all $k \in K$ such that $\mathrm{dim}(V_k)=\mathrm{dim}(V_{k_0})$.
\end{prp}
We do not know whether the assumption $\mathrm{dim}(V_k)=\mathrm{dim}(V_{k_0})$ is necessary. At least in the case of $K=\mathrm{SO(n)}$ and $U=\mathrm{SO}(n-1) \times \mathrm{SO}(n-1)$ (Lemma \ref{lem=coefficients_nonbi-invariant_Banach_rep_SOn}) of the case of \cite{delaatdelasalle1} the proposition holds without this assumption, because the norm in $B(L^2(K))$ of $\int_{U} (\lambda(u \cdot k) - \lambda(u \cdot k_0))du$ is greater than $1$ otherwise. 

We will use the following lemma.
\begin{lem}\label{lem:psi} Let $K,k_0,U$ and $V$ be as above. Then there exists a Lipschitz map $\psi \colon K \to B(V)$ such that 
\begin{enumerate}
\item\label{item:equivariance} $\psi(u \cdot k) = \psi(k) \rho(u^{-1})$ for every $u \in U$,
\item\label{item:coeff} for $v_1,v_2 \in V$, the functions $\langle \psi(\cdot) v_1,v_2\rangle$ are coefficients of $\lambda$,
\item\label{item:id_V0} $\psi(k_0) = P_{k_0}$.
\end{enumerate}
\end{lem}
\begin{proof}
Denote by $F$ the set of functions $\phi \colon K \to B(V)$ such that for all $v_1,v_2 \in V$, the function $\langle \phi(\cdot) v_1,v_2\rangle$ for $v_1,v_2 \in V$ is a coefficient of a finite-dimensional representation of $K$. Note that every function in $F$ is  $C^\infty$ (and hence Lipschitz) and satisfies \eqref{item:coeff} by the Peter-Weyl theorem. For a continuous function $\phi \colon K \to B(V)$, consider the function $\psi(k) = \int_U \phi(u\cdot k)\rho(u) du$. Then $\psi \in F$ if $\varphi \in F$ and $\psi(u \cdot k) =\psi(k) \rho(u^{-1})$. In particular, for $u \in U_{k_0}$ we have $\psi(k_0) = \psi(k_0) \rho(u^{-1}) $ and $\psi(k_0)$ vanishes on the orthogonal complement of $V_{k_0}$. We claim that there is a choice of $\phi \in F$ such that $\psi(k_0)$ has rank $\mathrm{dim(V_{k_0})}$. Before we prove the claim, let us explain how it implies the lemma. First, by replacing $\phi$ by $A \phi$ for some $A \in B(V)$ satisfying $A \psi(k_0)=P_{k_0}$ (such an $A$ exists because $\psi(k_0)$ has the same kernel as $P_{k_0}$), we can assume that $\psi(k_0) = P_{k_0}$, so that (\ref{item:id_V0}) holds. We already explained that $\psi$ is Lipschitz and that (\ref{item:equivariance}) and (\ref{item:coeff}) hold. Hence $\psi$ satisfies all the conditions in the lemma.

Let us now prove the claim. Denote by $O \subset K$ the $U$-orbit of $k_0$ (this is a closed subset of $K$), and let $s \colon O  \to U$ be a measurable section, i.e.~a measurable map satisfying $s(u \cdot k_0) \in u U_{k_0}$ for every $u \in U$. Then if $\phi \colon O \to B(V)$ is defined by $\phi(x) = \rho( s(x))^{-1}$, we see that $\int \phi(u \cdot k_0) \rho(u) du = \int  \rho( s(u \cdot k_0)^{-1} u) du$ acts as the identity on $V_{k_0}$, and therefore, since it vanishes on $V_{k_0}^\perp$, it is equal to $P_{k_0}$. By a density argument, this implies that there is a continuous function $\phi \colon O \to B(V)$ such that $\int \phi(u \cdot k_0) \rho(u) du$ is arbitrarily close to $P_{k_0}$, and in particular it has rank $\mathrm{dim}(V_{k_0})$. By the Tietze extension theorem, we can extend $\phi$ to a continuous function on $K$. For this $\phi$, it holds that $\psi(k_0)$ has rank $\mathrm{dim}(V_{k_0})$. By the density of $F$ in the space of continuous functions from $K$ to $B(V)$, the claim follows. 
\end{proof}
The next lemma is where the assumption  $\mathrm{dim}(V_k)=\mathrm{dim}(V_{k_0})$ is used.
\begin{lem}\label{lem=P_k_Lipschitz} There exists a constant $C_{V,k_0}>0$ such that 
\[\| P_k - P_{k_0} \|\leq C_{V,k_0} d(k,k_0)\]
for every $k$ such that $\mathrm{dim}(V_{k_0}) = \mathrm{dim}(V_{k})$.
\end{lem}
\begin{proof}
Let $\psi$ be a function given by Lemma \ref{lem:psi}, with Lipschitz constant $Lip(\psi)$. Take $k$ such that $\mathrm{dim}(V_{k_0}) = \mathrm{dim}(V_{k})$; denote by $d$ this common dimension. By (\ref{item:equivariance}) in Lemma \ref{lem:psi}, $\psi(k)$ vanishes on the orthogonal complement of $V_k$, and therefore $\psi(k)$ has rank at most $d$. If $d(k,k_0)\geq ( 10 Lip(\psi))^{-1}$ the inequality is obvious with $C_{V_0,k}= 20 Lip(\psi)$ because $\| P_k - P_{k_0} \| \leq 2$. We can therefore assume that $d(k,k_0)< ( 10 Lip(\psi))^{-1}$. This implies that $\|\psi(k) - P_{k_0}\|<1/10$, and therefore $\|\psi(k)^* \psi(k) - P_{k_0}\| < \frac 1 3$. This implies that the self-adjoint matrix $\psi(k)^* \psi(k)$ has $d$ eigenvalues in the interval $[\frac{2}{3}, \frac{4}{3}]$ and $\mathrm{dim}(V)-d$ eigenvalues in the interval $[0,\frac{1}{3}]$.  By a rank consideration the $\mathrm{dim}(V)-d$ smallest eigenvalues vanish, and the eigenvectors corresponding to the $d$ eigenvalues in $[\frac{2}{3}, \frac{4}{3}]$ span $V_k$. If $f$ is a Lipschitz function on $\R$ which is equal to $0$ at $0$ and $1$ on $[\frac{2}{3},\frac{4}{3}]$ we therefore have $f(\psi(k)^* \psi(k)) = P_k$ and $f(P_{k_0}) = P_{k_0}$. The conclusion follows because $A \mapsto f(A)$ is Lispschitz on the self-adjoint linear maps on $V$.
\end{proof}

\begin{proof}[Proof of Proposition \ref{prop:nonequivariant}] If $V$ is the trivial representation, we have
\[
  \|f(\pi(k)) - f(\pi(k_0))\|_V \leq \left\| \left(\int_{U} (\lambda(u \cdot k) - \lambda(u \cdot k_0))du \right) \otimes \mathrm{Id}_X \right\|_{B(L^2(K;X))}
\] 
for all $k \in K$ by \cite[Proposition 2.7]{delasalle1}.

If $V$ is not the trivial representation, we can reduce to the case of the trivial representation. Indeed, let $\psi$ be as in Lemma \ref{lem:psi}. Since $\psi(u \cdot k) = \psi(k) \rho(u^{-1})$ for all $k \in K$ and $u \in U$, the map from $K$ to $V$ given by $k \mapsto \psi(k) f(\pi(k))$ is $U$-invariant and we claim that it can be decomposed as $k \mapsto \sum_{i=1}^m \langle (\lambda \otimes \pi)(k) \tilde{\xi}_i, \tilde{\eta}_i\rangle w_i$ for $\sum_{i=1}^m \|\tilde{\xi}_i\|_{L^2(K;X)} \|\tilde{\eta}_i\|_{L^2(K;X^*)} \|w_i\|_V \leq C$, where $C$ depends on $\psi$ only. This decomposition can be obtained by taking an orthonormal basis $(e_1,\dots,e_d)$ of $V$ and writing each function $\langle \psi(\cdot) v_i, e_j\rangle$ as a coefficient of the left-regular representation $\lambda$ of $K$ as follows:~$\langle \psi(\cdot) v_i, e_j\rangle= \langle \lambda(\cdot) a_{i,j},b_{i,j}\rangle$ for some $a_{i,j},b_{i,j} \in L^2(K)$ of norm less than $(C \|v_i\|)^{\frac 1 2}$. This gives that $\psi(k) v_i = \sum_{j=1}^d \langle \lambda(\cdot) a_{i,j},b_{i,j}\rangle e_j$ and hence 
\[ \psi(k) f(\pi(k)) = \sum_{i=1}^n\sum_{j=1}^d \langle (\lambda \otimes \pi)(k) a_{i,j} \otimes \xi_i,b_{i,j} \otimes \eta_j\rangle e_j,\]
which is of the announced form with $m=nd$. Observe also that for all $k \in K$,
\begin{equation}\label{eq:fk_in_Vk} f(\pi(k)) \in V_k\end{equation} because for every $u \in U_k$, $\rho(u) f(\pi(k)) = f(\pi(u \cdot k)) = f(\pi(k))$. Therefore, by the previous case for the representation $\lambda \otimes \pi$ and by \eqref{item:id_V0} of Lemma \ref{lem:psi},
\begin{equation} \label{eq=psi}
\begin{split}
  \| \psi(k)f(\pi(k)) - &f(\pi(k_0))\|_V \leq \\ &C \left\| \left(\int_{U} (\lambda(u \cdot k) - \lambda(u \cdot k_0))du\right) \otimes \mathrm{Id}_X \right\|_{B(L^2(K;X))}
\end{split}
\end{equation}
for all $k \in K$. By using the triangle inequality and \eqref{eq:fk_in_Vk}, we obtain
\[
  \|f(\pi(k)) - f(\pi(k_0))\|_V \leq \| \psi(k)f(\pi(k)) - f(\pi(k_0))\|_V + \|\psi(k)-P_k\|_{B(V)} \|f(\pi(k))\|_V
\]
for all $k \in K$. Estimating the first term by \eqref{eq=psi}, using the fact that $\psi$ is a Lipschitz map and Lemma \ref{lem=P_k_Lipschitz}, which implies that 
\[\|\psi(k)-P_k\|_{B(V)} \leq \|\psi(k)-\psi(k_0)\|_{B(V)} + \|P_{k_0}-P_k\|_{B(V)} \leq C'd(k,k_0)\] for some constant $C'>0$, and using the inequality $\|f(\pi(k))\| \leq 1$ for all $k \in K$, the proposition follows. 
\end{proof}
\section{Strong property (T), property (F$_X$) and property (FF$_X$)} \label{sec=properties}
Recall that in this article, representations of locally compact groups on Banach spaces are always assumed to be linear and strongly continuous, unless explicitly stated otherwise.

\subsection{Strong property (T)}\label{subsect:strongTdef}
Strong property (T) was introduced by Lafforgue as an obstruction to a certain approach to the Baum--Connes Conjecture \cite{lafforguestrongt,lafforguefastfourier}. It is defined in terms of the existence of a self-adjoint idempotent (the Kazhdan projection) in a certain completion of the space of compactly supported continuous functions on the group. In this article, we use the following more flexible definition of strong property (T), as introduced by the third-named author in \cite{delasalle1}. Recall first that a length function on a locally compact group $G$ is a continuous function $\ell \colon G \to \R_+$ such that $\ell(g_1^{-1})= \ell(g_1)$ and $\ell(g_1g_2)\leq \ell(g_1)+\ell(g_2)$ for all $g_1,g_2 \in G$.
\begin{dfn} \label{def=strongt}
A locally compact group $G$ has strong property (T) with respect to a class $\mathcal{E}$ of Banach spaces (denoted by (T$^{\mathrm{strong}}_{\mathcal{E}}$)) if for every length function $\ell$ on $G$ there is a sequence of compactly supported symmetric Borel measures $m_n$ on $G$ such that for every Banach space $X$ in $\mathcal E$ there exists a constant $t>0$ such that the following holds: for every representation $\pi:G \rightarrow B(X)$ with $\|\pi(g)\|_{B(X)} \leq L e^{t \ell(g)}$ for some $L\in\R_+$, the sequence $\pi(m_n)$ converges in the norm topology on $B(X)$ to a projection onto the $\pi(G)$-invariant vectors of $X$.
\end{dfn}
If one takes $\mathcal{E}$ to be the class of Hilbert spaces, this definition is equivalent to Lafforgue's original definition of strong property (T), which is denoted by (T$^{\mathrm{strong}}_{\mathrm{Hilbert}}$). If $\mathcal{E}$ is taken to be the class of Banach spaces with nontrivial type, the above definition is equivalent to Lafforgue's strong Banach property (T), denoted (T$^{\mathrm{strong}}_{\mathrm{Banach}}$). The equivalences are shown in \cite[Section 2.9]{delasalle1}. We sometimes use the notation (T$^{\mathrm{strong}}_{X}$) for strong property (T) with respect to the isomorphism class of $X$.

The work in \cite{lafforguestrongt} implies that connected simple Lie groups with real rank $1$ do not have (T$^{\mathrm{strong}}_{\mathrm{Hilbert}}$). Furthermore, Lafforgue showed that $\mathrm{SL}(3,\R)$ has (T$^{\mathrm{strong}}_{\mathrm{Hilbert}}$), and the third-named author proved that $\mathrm{SL}(3,\R)$ has (T$^{\mathrm{strong}}_{\mathcal{E}_4}$) in \cite{delasalle1}, where, for $r > 2$, the class $\mathcal{E}_r$ is a certain class of Banach spaces containing the Hilbert spaces, many superreflexive spaces and some non-reflexive spaces. In \cite{delaatdelasalle1}, the first-named author and the third-named author extended this result, by proving that connected simple Lie groups with real rank at least $2$ have (T$^{\mathrm{strong}}_{\mathcal{E}_{10}}$).

\subsection{Property (F$_X$)}
\begin{dfn}
  Let $X$ be a Banach space. A locally compact group $G$ has property (F$_X$) if every continuous affine isometric action of $G$ on $X$ has a fixed point.
\end{dfn}
This property was introduced in \cite{baderfurmangelandermonod}. In that article, it is proved that if $G$ is a connected simple Lie group with real rank at least $2$ and finite center or a lattice in such a group, then $G$ has property (F$_X$) for every subspace or quotient $X$ of an arbitrary $L^p$-space, where $1<p<\infty$ (see \cite[Theorem B]{baderfurmangelandermonod} for a more general and precise statement). This result and its proof were the motivation for the conjecture of Bader, Furman, Gelander and Monod mentioned in Remark \ref{rem=fx}, (i).

We now explain that Theorem \ref{thm=fpp} implies Corollary \ref{cor=fpplattices}.
\begin{proof}[Proof of Corollary \ref{cor=fpplattices}]
Let $G$ be a locally compact group, and let $\Gamma$ be a lattice in $G$, i.e.~a discrete subgroup with finite invariant measure. It was proved in \cite[Proposition 8.8]{baderfurmangelandermonod} that (F$_{X}$) for $\Gamma$ follows from (F$_{L^2(G/\Gamma;X)}$) for $G$ provided that $\Gamma$ is a so-called $2$-integrable lattice in $G$ (see \cite{baderfurmangelandermonod} or \cite{delaatdelasalle1} for definitions). Also, Shalom proved that all lattices in a higher rank algebraic group are $2$-integrable, and this was extended to simple Lie groups of higher rank in \cite[Proposition 7.1]{delaatdelasalle1}. Therefore, all we have to show to justify the implication Theorem \ref{thm=fpp} $\implies$ Corollary \ref{cor=fpplattices} is that if $\beta<\frac{1}{2}$ and $X$ is a Banach space satisfying \eqref{eq=d_kX_grows_slowly}, then also $L^2(\Omega,\mu;X)$ satisfies \eqref{eq=d_kX_grows_slowly} for every probability measure space $(\Omega,\mu)$. This is certainly well-known. One quick argument is to use inequality $(6)$ from \cite{delaatdelasalle2} which, together with the Fubini Theorem, implies that $d_k(L^2(\Omega,\mu;X)) \leq 2 d_k(X)$ for all $k \geq 1$.
\end{proof}

\subsection{Property (FF$_X$)} \label{subsec=ffx}
\begin{dfn}\label{def=quasicocycle}
Let $G$ be a locally compact group,  let $X$ be a Banach space, and let $\rho \colon G \to O(X)$ be an isometric representation. A continuous map $c\colon G\to X$ is called a quasi-$1$-cocycle into $\rho$ if 
\[
\sup_{g_1,g_2\in G}\|c(g_1g_2)-c(g_1)-\rho(g_1)c(g_2)\| <\infty.
\]
The above quantity is called the defect of the quasi-$1$-cocycle $c$.
\end{dfn}
\begin{dfn}\label{def=FFX}
A locally compact group $G$ is said to have property (FF$_X$) if for every isometric representation $\rho \colon G \to O(X)$, every quasi-$1$-cocycle $c\colon G\to X$ into $\rho$ is bounded.
\end{dfn}
Let us recall the relation of property (FF$_X$) to bounded cohomology. It is known that a locally compact group $G$ has property (FF$_X$) if and only if the following two statements are satisfied: 
\begin{enumerate}[(i)]
  \item for every isometric representation $\rho \colon G \to O(X)$, the comparison map 
\[
\Psi^2_{\mathrm{c}}\colon H^2_{\mathrm{cb}}(G;X,\rho) \to H^2_{\mathrm{c}}(G;X,\rho)
\]
from the second continuous bounded cohomology group to the second continuous cohomology group with coefficients in $(X,\rho)$ is injective; 
  \item for every continuous action of $G$ on $X$ by affine isometries, some (or equivalently every) $G$-orbit is bounded. 
\end{enumerate}
We refer to item (ii) in \cite[Corollary~13.1.10]{monod} for the proof of the above equivalence, and we refer the reader to \cite{monodicm} and \cite[Chapter~V.13]{monod} for a comprehensive treatment of quasi-cocycles and bounded cohomology.

Generally, for a Banach space $X$, the boundedness property (statement (ii) above) may be weaker than property (F$_X$). However, if $X$ satisfies the property that every continuous group action on $X$ by isometries with bounded orbits has a global fixed point, for instance, if $X$ is reflexive, then it is equivalent to the fixed point property (recall Remark~\ref{rem=fpffp}). We in addition mention that property (FF$_X$) is, in general, strictly stronger than property (F$_X$) for such an $X$. For example, certain infinite hyperbolic groups have property (T), and hence property (FH). On the other hand, it is well known that any infinite hyperbolic group fails to have property (TT) of Monod (see \cite{epsteinfujiwara} and \cite{mineyevmonodshalom}), which is the Hilbert space version of property (FF$_X$).

In our arguments, we need some more notions and certain relative versions of the properties recalled above. Firstly, we recall the following notion from \cite{baderfurmangelandermonod} (see also \cite{mimura}). Let $G$ be a compactly generated locally compact group (note that every connected locally compact group is compactly generated), and let $S$ be a compact generating set of $G$. Recall that an isometric representation $\pi \colon G \to O(Y)$ does \emph{not} have almost invariant vectors if there exists an $\epsilon=\epsilon(G,S,\pi)>0$ such that for any $\eta \in Y\setminus \{0\}$, we have $\sup_{s\in S}\|\pi(s)\eta-\eta\|>\epsilon \|\eta\|$.

\begin{dfn}\label{def=ffpp}
Let $G$ be a compactly generated locally compact group with compact generating set $S$, let $H$ be a closed subgroup of $G$, and let $X$ be a Banach space.
\begin{enumerate}
  \item Suppose, in addition, that $H$ is a normal subgroup of $G$. The pair \linebreak $G \triangleright H$ is said to have relative property (T$_X$) if the following condition is satisfied: for every isometric representation $\rho \colon G \to O(X)$, the corresponding induced representation $\overline{\rho}\colon G \to O(X/X^{\rho(H)})$ does not have almost invariant vectors. Here $X^{\rho(H)}$ is the subspace of $X$ consisting of all $\rho(H)$-invariant vectors. (Indeed, because $H$ is normal, $\rho$ induces an isometric representation on $X/X^{\rho(H)}$, which we write as $\overline{\rho}$.) We define $\kappa(G,H,S,\rho)$ as the supremum of the constants $\epsilon(G,S,\overline{\rho})$.

  The group $G$ has property (T$_X$) if $G \triangleright G$ has relative property (T$_X$).
  \item We say that the pair $G > H$ is said to have weak relative property (T$_X$) if the following holds: for any isometric representation $\rho \colon G \to O(X)$, there exists a strictly positive constant $\tilde{\epsilon}$ such that for any $\xi \in X$,
\[
 \sup_{s\in S} \|\rho(s)\xi -\xi\| \geq \tilde{\epsilon} \sup_{h\in H} \|\rho(h)\xi -\xi\|.
\]
We set $\tilde{\kappa}(G,H,S,\rho)$ as the supremum of such $\tilde{\epsilon}$.
  \item The pair $G>H$ has relative property (FF$_X$) if for every isometric representation $\rho \colon G \to O(X)$, any quasi-$1$-cocycle $c\colon G\to X$ into $\rho$ is bounded on $H$.
  
  The group $G$ has property (FF$_X$) in the sense of Definition~\ref{def=FFX} if $G>G$ has relative property (FF$_X$).
\end{enumerate}
\end{dfn}
Let us point out that for a reflexive Banach space $X$, relative property (T$_X$) for $G \triangleright H$ coincides with weak relative property (T$_X$) for $G \triangleright H$. Indeed, in that case, for any isometric representation $\rho\colon G\to O(X)$ and for any $\xi \in X$, the inequality
\[
  2 \|\overline{\xi}\|_{X/X^{\rho(H)}} \geq \sup_{h\in H} \|\rho(h)\xi -\xi\|_X \geq \|\overline{\xi}\|_{X/X^{\rho(H)}}
\]
holds, where $\xi\mapsto \overline{\xi}$ denotes the quotient map $X\twoheadrightarrow X/X^{\rho(H)}$. The inequality on the right side uses the reflexivity assumption on $X$ (see Lemma~\ref{lem=dual}). On the other hand, the inequality on the left side always holds, and hence relative property (T$_X$) implies weak relative property (T$_X$) for any $X$. More precisely, $2\tilde{\kappa}(G,H,S,\rho)\geq \kappa(G,H,S,\rho) \geq \tilde{\kappa}(G,H,S,\rho)$ holds as long as $X$ is reflexive, and the inequality on the left side holds without any assumption on $X$. In general case, however, we do not know whether the latter property is strictly weaker. In Proposition~\ref{prp=fromttoffpp}, which is a key proposition in Section~\ref{sec=ffpp}, we only need to assume the latter property.

We remark that the properties above are independent of the choices of $S$, despite the fact that the constants $\kappa(G,H,S,\rho)$ and $\tilde{\kappa}(G,H,S,\rho)$ depend on $S$.

\section{Strong property (T) for high rank Lie groups} \label{sec=strongt}
In this section, we prove Theorem \ref{thm=strongt}. The main point is to prove the theorem for $G=\mathrm{SL}(N,\R)$ (Proposition \ref{prp=strongtsln}), in which case the measures $m_n$ appearing in Definition \ref{def=strongt} are the uniform measures on the sets $K g_n K$, where $K=\mathrm{SO}(N)$ is the maximal compact subgroup of $G$ and $(g_n)$ is a well-chosen sequence in $G$ tending to infinity. The fact that the sequence $\pi(m_n)$ has a limit in the norm topology for a specific sequence $g_n$ was essentially already proved in \cite{delaatdelasalle2}. The fact that, unlike in the rank two case (see \cite{lafforguestrongt,lafforguefastfourier,liao,delaatdelasalle1}), we cannot take any sequence $(g_n)$ tending to infinity, makes our task of identifying the limit with a projection onto the invariant vectors significantly harder. The main technical improvement compared to \cite{delaatdelasalle2} is that, at the cost of increasing the value of $N$, we manage to prove the convergence of $\pi(m_n)$ under much weaker conditions on the sequence $(g_n)$ (see \eqref{eq=convergence_in_D}). In particular, this allows us to show that $\lim_{k,n \to \infty} \pi(m_k \ast m_n)$ also converges to $\lim_n \pi(m_n)$, which was the first obstacle encountered in \cite{delaatdelasalle2} to prove strong property (T).

Fix $n \geq 3$. In what follows, operators, functions and constants may implicitly depend on $n$. Considering $\mathrm{SO}(n-1)$ as the subgroup of $\mathrm{SO}(n)$ fixing the first coordinate vector $e_1 \in \R^n$, and using the identification $\sphere^{n-1} \cong \mathrm{SO}(n-1) \backslash \mathrm{SO}(n)$ through $\mathrm{SO}(n-1) k \mapsto k^{-1} e_1$, we can view $L^2(\sphere^{n-1})$ as a subspace of $L^2(\mathrm{SO}(n))$. Let $\lambda$ denote the left-regular representation, and define the operator $T_\delta$ on $L^2(\mathrm{SO}(n))$ by $T_{\delta} = \int_{\mathrm{SO}(n-1) \times \mathrm{SO}(n-1)} \lambda( u k u') du du'$ for $k \in \mathrm{SO}(n)$ satisfying $k_{11} = \delta$. Here, $k_{11}$ denotes the entry in the first row and first column of $k$. The operator $T_{\delta}$ does not depend on the choice of $k$, since $\int_{\mathrm{SO}(n-1) \times \mathrm{SO}(n-1)} \lambda( u k u') du du'$ depends only on $k_{11}$. Note that $T_\delta$ maps the subspace $L^2(\sphere^{n-1})$ to itself and is zero on its orthogonal complement, so that $T_\delta$ can be considered as an operator on $L^2(\sphere^{n-1})$. For more details on and related use of the operator $T_{\delta}$, we refer to \cite{lafforguestrongt,lafforguedelasalle,delasalle1,delaatdelasalle1,delaatdelasalle2}. The following result is \cite[Lemma 5.4]{delaatdelasalle2}.
\begin{lem}\label{lem=estimates_Tdelta_X}
Let $n \geq 3$, and let $X$ be a Banach space for which there exist $\beta < \frac 1 2(1- \frac{1}{n-1})$ and $C'>0$ such that $d_k(X) \leq C' k^\beta$ for all $k \geq 1$. Then there exist $\alpha_X \in (0,1)$ and $C_X>0$ such that for all $\delta,\delta' \in [-1,1]$,
\begin{equation}\label{eq=ine_Tdelta_X} \| (T_{\delta} - T_{\delta'})\otimes \mathrm{Id}_X\|_{B(L^2(\mathrm{SO}(n);X))} \leq C_X \max(|\delta|^{\alpha_X}, |\delta'|^{\alpha_X}).\end{equation}
Moreover, $\alpha_X$ and $C_X$ depend on $\beta$ and $C'$ only.
\end{lem}
The way we apply the preceding lemma is through the following result, which is an extension of \cite[Lemma 5.6]{delaatdelasalle2} and follows from Proposition \ref{prop:nonequivariant}.
\begin{lem} \label{lem=coefficients_nonbi-invariant_Banach_rep_SOn} Let $n \geq 3$, and let $X$ be a Banach space satisfying \eqref{eq=ine_Tdelta_X} for some $\alpha_X \in (0,1)$ and $C_X>0$, and all $\delta,\delta' \in [-1,1]$. Let $\tilde{k}_0$ denote an element of $\mathrm{SO}(n)$ with $(\tilde{k}_0)_{11}=0$ (the upper left entry of $\tilde{k}_0$ equals $0$). Suppose that $V$ is a finite-dimensional unitary representation of $\mathrm{SO}(n-1)$. Then there exists a constant $C(X,V)$ (depending on $X$ and $V$) such that for all isometric representations $\pi \colon \mathrm{SO}(n) \to O(X)$, all $\mathrm{SO}(n-1)$-invariant vectors $\xi\in X$ and all $\mathrm{SO}(n-1)$-equivariant linear maps $q \colon X \to V$, the function $\varphi$ given by $\varphi(\tilde{k}) = q(\pi(\tilde{k})\xi)$ satisfies
\[
  \|\varphi(\tilde{k}) - \varphi(\tilde{k}_0)\|_V \leq C(X,V)(d(\tilde{k},\tilde{k}_0)+ |\tilde{k}_{11}|^{\alpha_X}) \|\xi\|_X \|q\|_{X \to V}
\]
for all $\tilde{k} \in \mathrm{SO}(n)$. Here $d$ denotes some fixed distance on $\mathrm{SO}(n)$ coming from an invariant Riemannian metric.
\end{lem}
\begin{proof} If $|\widetilde k_{1,1}|=1$ there is nothing to prove, so we can assume that $|\widetilde k_{1,1}|<1$. Apply Proposition \ref{prop:nonequivariant} with $K= \mathrm{SO}(n)$, $U = \mathrm{SO}(n-1) \times \mathrm{SO}(n-1)$, the representation $\rho$ of $U$ on $V$ given by $\rho(u,u') v=u\cdot v$ for $(u,u') \in U$, and the function $f(a) = q( a \xi)$. The assumption that $\mathrm{dim}(V_{\tilde k}) = \mathrm{dim}(V_{\tilde k_0})$ holds because if  $|\tilde k_{1,1}|<1$, the groups $U_{\tilde k}$ and $U_{\tilde k_0}$ are both conjugate to the subgroup $\mathrm{SO}(n-2)$ (as a diagonal subgroup of $\mathrm{SO}(n-2)\times \mathrm{SO}(n-2) \subset U$). The term 
\[\left\| \left(\int_{U} (\lambda(u \cdot k) - \lambda(u \cdot k_0))du\right) \otimes \mathrm{Id}_X \right\|_{B(L^2(K;X))}\] is equal to
\[\| (T_{\widetilde k_{11}} - T_{0})\otimes \mathrm{Id}_X\|_{B(L^2(\mathrm{SO}(n);X))},\]
which is less than $C_X |\widetilde k_{11}|^{\alpha_X}$ by Lemma \ref{lem=estimates_Tdelta_X}. 
\end{proof}
For $\delta \in [-1,1]$, let $\tilde{k}_\delta \in \mathrm{SO}(n)$ denote a rotation by angle $\theta=\arccos(\delta) \in [0,\pi]$ in the plane generated by the first two coordinate vectors of $\R^n$. Then $(\tilde{k}_\delta)_{11}=\delta$ and $d(\tilde{k}_\delta,\tilde{k}_0) = \mathcal{O}(|\delta|)$, and the corresponding special case of the inequality of Lemma \ref{lem=coefficients_nonbi-invariant_Banach_rep_SOn} is
\begin{equation} \label{eq=coeff_noninv_SOn}
  \|\varphi(\tilde{k}_\delta) - \varphi(\tilde{k}_0)\|_V \leq C(X,V)|\delta|^{\alpha_X} \|\xi\|_X \|q\|_{X \to V},
\end{equation}
perhaps with a different constant $C(X,V)$. In particular,
\begin{equation} \label{eq=coeff_norm_noninv_SOn}
  \left| \|\varphi(\tilde{k}_\delta)\|_V - \|\varphi(\tilde{k}_0)\|_V\right| \leq C(X,V)|\delta|^{\alpha_X} \|\xi\|_X \|q\|_{X \to V}.
\end{equation}
For an integer $N$, equip $\mathrm{SL}(N,\R)$ with the length function defined by $\ell(g) = \max(\log \|g\|,\log\|g^{-1}\|)$, where $\|\cdot\|$ denotes the operator norm for operators on the Euclidean space $\R^N$. We will use that 
\begin{equation} \label{eq=length_function}
  \ell(g) \leq N\log \|g^{-1}\| \textrm{ for all }g \in \mathrm{SL}(N,\R).
\end{equation}
Before we turn to strong property (T) for the group $\mathrm{SL}(N,\R)$, we first explain its polar decomposition. We briefly recalled the basics of this decomposition in Section \ref{subsec=kak}. It is well-known that the group $\mathrm{SO}(N)$ is a maximal compact subgroup of $\mathrm{SL}(N,\R)$. Therefore, $\mathrm{SL}(N,\R)$ admits a polar decomposition of the form $\mathrm{SL}(N,\R)=\mathrm{SO}(N) \cdot A \cdot \mathrm{SO}(N)$, and $A$ is given by
\[
  A=\{\mathrm{diag}(e^{\alpha_1},\ldots,e^{\alpha_N}) \mid \alpha_1 + \ldots + \alpha_N = 0\}.
\]
If we restrict to $\overline{A^{+}}=\{\mathrm{diag}(e^{\alpha_1},\ldots,e^{\alpha_N}) \mid \alpha_1 + \ldots + \alpha_N = 0,\;\alpha_1 \geq \ldots \geq \alpha_N\}$, we still have $\mathrm{SL}(N,\R)=\mathrm{SO}(N) \cdot \overline{A^{+}} \cdot \mathrm{SO}(n)$, and moreover, the element $a \in \overline{A^{+}}$ is uniquely determined.

Let $m \in \N$. In what follows, for $r,s,t \in \R$ satisfying $r+s+t=0$, we denote by $D(r,s,t)$ the diagonal matrix in $\mathrm{SL}(3m,\R)$ with $m$ eigenvalues equal to $e^r$, and $m$ ones equal to $e^s$, and the other $m$ eigenvalues equal to $e^{t}$, in such a way that the eigenvalues are ordered as follows: $D(r,s,t)=\mathrm{diag}(e^r,\ldots,e^r,e^s,\ldots,e^s,e^t,\ldots,e^t)$.

For $u > 0$, let $m_u$ denote the compactly supported probability measure on $\mathrm{SL}(3m,\R)$ defined by
\[
  m_u(f) = \int_{\mathrm{SO}(3m)}\int_{\mathrm{SO}(3m)} f(k D(u,0,-u) k') dk dk',
\]
where $dk$ and $dk'$ denote the normalized Haar measure on $\mathrm{SO}(3m)$. In the following proposition, we consider a certain sequence of such measures. In the following, we will take $m=n-2$, where $n \geq 3$.

The main task in the rest of this section is to prove the following proposition, which implies Theorem \ref{thm=strongt}.
\begin{prp} \label{prp=strongtsln}
Let $n \geq 3$, and let $(m_u)_{u \in \R^+}$ be the family of measures on $\mathrm{SL}(3n-6,\R)$ as defined above. Then for every Banach space satisfying \eqref{eq=ine_Tdelta_X} and every representation $\pi \colon \mathrm{SL}(3n-6,\R) \to B(X)$ satisfying $\sup_{g \in G} e^{-\gamma \ell(g)} \|\pi(g)\| < \infty$ for some $\gamma < \frac {\alpha_X}{3n-6}$, there exists an operator $P$ in $B(X)$ such that 
\begin{enumerate}[(i)]
\item \label{item=P_projection} the operator $P$ is a projection onto the subspace $X^{\pi(G)}$ of $\pi(G)$-invariant vectors in $X$,
\item \label{item=P_limit} we have $\lim_{u \to \infty} e^{(\frac{\alpha_X}{2} - (3n-6)\gamma) u} \|\pi(m_u) - P\|_{B(X)} = 0$.
\end{enumerate}
\end{prp}
Note that the measure $m_u$ is just the uniform probability measure on the subset $\mathrm{SO}(3n-6) \cdot D(u,0,-u) \cdot \mathrm{SO}(3n-6)$ of $\mathrm{SL}(3n-6,\R)$.
\begin{proof}[Proof of Theorem \ref{thm=strongt} using Proposition \ref{prp=strongtsln}]
Let $\beta < \frac{1}{2}$, and take $n \geq 3$ to be the smallest natural number such that $\beta < \frac{1}{2}(1-\frac{1}{n-1})$. Put $N=3n-6$. It follows directly from Proposition \ref{prp=strongtsln} and Lemma \ref{lem=estimates_Tdelta_X} that $\mathrm{SL}(N,\R)$ has strong property (T) with respect to the Banach spaces satisfying \eqref{eq=d_kX_grows_slowly}.

It follows from a result of Dynkin in \cite{dynkin} (see \cite{dynkinselectedworks} for a translation) that every connected simple Lie group with real rank $N \geq 9$ contains a closed analytic subgroup locally isomorphic to $\mathrm{SL}(N,\R)$ (see \cite[Lemma 4.6]{delaatdelasalle2} for more details). By a similar argument as in the proof of \cite[Corollaire 4.1]{lafforguestrongt} (see also the proof of \cite[Theorem A]{delaatdelasalle1}), strong property (T) with respect to the Banach spaces satisfying \eqref{eq=d_kX_grows_slowly} now follows for connected simple Lie groups with real rank at least $N$.
\end{proof}

\subsection*{Proof of Proposition \ref{prp=strongtsln}} The rest of this section is devoted to the proof of Proposition \ref{prp=strongtsln}. The proof follows the same strategy as the proof of strong property (T) for $\mathrm{SL}(3,\R)$ relative to the class of Hilbert spaces, which is due to Lafforgue \cite[Section 2]{lafforguestrongt} (see also \cite{delasalle1} for strong property (T) for this group relative to certain Banach spaces). However, identifying the operator $P$ with the projection onto the $\pi(G)$-invariant vectors in $X$ involves new ideas and techniques.

From now on, we fix $n$, $X$, $\pi$ and $\gamma$ as in the statement of Proposition \ref{prp=strongtsln}. We put $G=\mathrm{SL}(3n-6,\R)$ and $K=\mathrm{SO}(3n-6)$. By renorming $X$ with the equivalent norm given by assigning the number $\sup_{k \in K} \|\pi(k)x\|$ to $x \in X$, we may assume that the restriction of $\pi$ to $\mathrm{SO}(3n-6)$ is isometric. We still denote this norm by $\|.\|$. Let $L = \sup_{g \in \mathrm{SL}(3n-6,\R)}e^{-\gamma \ell(g)}\|\pi(g)\|$, and for $g \in G$, let $\widetilde{m}_g$ denote the uniform probabilty measure on $KgK$, i.e.,
\[
  \widetilde m_g(f) = \int_K \int_K f(kgk') dk dk',
\]
so that the measures $m_u$ of the proposition are given by $m_u=\widetilde{m}_{D(u,0,-u)}$.

The proof now proceeds in three steps. The first step is a straightforward adaptation of an argument from \cite{delaatdelasalle2}.\\

{\underline{Step 1}}. The net of operators $(\pi(m_u))$ has a limit $P$ in the norm topology of $B(X)$ as $u \to \infty$ and conclusion \eqref{item=P_limit} of Proposition \ref{prp=strongtsln} holds.\\

As mentioned before, the difficulty is to identify $P$ as a projection onto the invariant vectors.\\

{\underline{Step 2}}. The operator $P$ is a projection.

{\underline{Step 3}}. The image of $P$ is the space $X^{\pi(G)}$ consisting of $\pi(G)$-invariant vectors.\\

For the proof of \underline{Step 1}, we first prove two lemmas, which generalize certain results from \cite[Section 4]{delaatdelasalle2} to the setting of representations that are not necessarily isometric.

\begin{lem} \label{lem=coeff_GLn}
For $K$-invariant unit vectors $\xi \in X$ and $\eta \in X^*$, set $\varphi(g) = \langle \pi(g) \xi,\eta \rangle$. Also, let $D=\mathrm{diag}(e^{v_1},\ldots,e^{v_{3n-6}}) \in \mathrm{SL}(3n-6,\R)$ be a diagonal matrix with $\min(v_1,\dots,v_{3n-6}) = v_{2n-3} = \ldots = v_{3n-6}=t$ for some $t<0$. Suppose $i<j \leq 2n-4$, and let $D'=\mathrm{diag}(e^{v_1'},\ldots,e^{v_{3n-6}'}) \in \mathrm{SL}(3n-6,\R)$ be another diagonal matrix with entries equal to those of $D$ except the $i$-th and the $j$-th entry, and also satisfying $v'_i,v'_j \geq t$. Then
\[
  |\varphi(D) - \varphi(D')|\leq 2 \widetilde{C}_X L^2 e^{-\alpha_X(\mu-t) - (3n-6)\gamma t},
\]
where $\mu=\min\{v_i,v_j,v'_i,v'_j\}$.
\end{lem}
\begin{proof} Let $\mathcal{A} = \mathrm{diag}(e^{a_1},\dots,e^{a_{3n-6}})$ be the matrix in $\mathrm{SL}(3n-6)$ with $a_i=\frac{v_i+v_j-t}{2}$, $a_j=\frac{t}{2}$ and $a_l=\frac{v_l}{2}$ if $l \notin \{i,j\}$. Using \eqref{eq=length_function} and $\min\{a_l \mid l=1,\ldots,3n-6\}=\frac{t}{2}$, we see that $\ell(\mathcal{A}) \leq -(3n-6)\frac{t}{2}$. 

Let $K' \subset K$ be the subgroup of $K$ that acts as the identity on the orthogonal complement of the space spanned by the $i$-th, the $j$-th and the last $n-2$ basis vectors of $\R^{3n-6}$. The group $K'$ is isomorphic to $\mathrm{SO}(n)$ and its subgroup $U'$ that fixes the $i$-th coordinate vector of $\R^{3n-6}$ is a copy of $\mathrm{SO}(n-1)$ contained in the centralizer of $\mathcal{A}$. In particular, the vectors $\pi(\mathcal{A}) \xi$ and $\pi(\mathcal{A})^* \eta$ are $U'$-invariant vectors of norm less than $\|\pi(\mathcal{A})\| \leq L e^{- (3n-6)\gamma\frac{t}{2}}$.

Let $k_\delta \in K$ be a rotation of angle $\arccos(\delta) \in [0,\pi]$ in the plane generated by the $i$-th and $j$-th coordinate vectors of $\R^2$. Notice that $k_\delta$ belongs to $K'$ and corresponds to the element $\tilde{k}_\delta \in \mathrm{SO}(n)$ under a correct identification of the pairs $(K',U')$ and $(\mathrm{SO}(n),\mathrm{SO}(n-1))$. Then \eqref{eq=coeff_noninv_SOn} implies that
\begin{equation} \label{eq=aka}
  |\varphi(\mathcal{A} k_\delta \mathcal{A}) - \varphi(\mathcal{A} k_0 \mathcal{A})| \leq \widetilde{C}_X|\delta|^{\alpha_X} L^2 e^{-(3n-6)\gamma t}
\end{equation}
for some constant $\widetilde{C}_X$ depending on $X$. Indeed, in the case that $V$ in \eqref{eq=coeff_noninv_SOn} is the trivial representation, which is exactly the case we are dealing with here, we denote the constant $C(X,V)$ by $\widetilde{C}_X$.

The matrix $\mathcal{A} k_\delta \mathcal{A}$ satisfies $\mathcal{A} k_\delta \mathcal{A} e_l = e^{2 a_l} e_l$ if $l \notin \{i,j\}$, and on the plane $\mathrm{span}\{e_i,e_j\}$, it acts as $B_{\delta}=\begin{pmatrix} e^{v_i+v_j-t} \delta & - e^{\frac{v_i+v_j}{2}} \sqrt{1-\delta^2} \\ e^{\frac{v_i+v_j}{2}} \sqrt{1-\delta^2} & e^{t} \delta\end{pmatrix}$.

By computation, it follows that there exists a $\delta \in [-1,1]$ such that the matrix $B_{\delta}$ for this value of $\delta$ belongs to $\mathrm{SO}(2) \mathrm{diag}(e^{v_i},e^{v_j}) \mathrm{SO}(2)$. This implies that $\mathcal{A} k_\delta \mathcal{A} \in K D K$. By a similar argument, there exists a $\delta' \in [-1,1]$ such that $\mathcal{A} k_{\delta'} \mathcal{A} \in KD'K$, so that by the $K$-bi-invariance of $\varphi$ and using the triangle inequality, we obtain
\[
  |\varphi(D) - \varphi(D')| \leq \widetilde{C}_X (|\delta|^{\alpha_X} + |\delta'|^{\alpha_X}) L^2 e^{-(3n-6)\gamma t}.
\]
Since the upper left entry of $\begin{pmatrix} e^{v_i+v_j-t} \delta & - e^{\frac{v_i+v_j}{2}} \sqrt{1-\delta^2}\\ e^{\frac{v_i+v_j}{2}} \sqrt{1-\delta^2} & e^{t} \delta\end{pmatrix}$ is smaller than or equal to the norm of this matrix, we get the inequality $e^{v_i+v_j-t}|\delta|\leq e^{\max(v_i,v_j)}$, and in particular $|\delta|\leq e^{t-\mu}$. Similarly, $|\delta'|\leq e^{t-\mu}$. This proves the lemma.
\end{proof}
\begin{lem}\label{lem=technical1}
For $K$-invariant unit vectors $\xi \in X$ and $\eta \in X^*$, set $\varphi(g) = \langle \pi(g) \xi,\eta \rangle$ as above. Then for all $u>0$ and $\delta>0$ satisfying $\delta < u$, 
\[
  |\varphi(D(u,0,-u)) - \varphi(D(u+\delta,-\delta,-u))| \leq (2n-4) \widetilde{C}_X L^2 e^{-\alpha_X(u-\delta) + (3n-6)\gamma u}.
\]
\end{lem}
\begin{proof}
  This lemma follows directly by applying Lemma \ref{lem=coeff_GLn} $n-2$ times.
\end{proof}
Now we can prove \underline{Step 1}.
\begin{proof}[Proof of \underline{Step 1}.] As a particular case of Lemma \ref{lem=technical1}, for $|\delta| \leq \frac{u}{2}$, we have
\[
  |\varphi(D(u,0,-u)) - \varphi(D(u+\delta,-\delta,-u))| \leq (2n-4) \widetilde{C}_X L^2 e^{-(\frac{\alpha_X}{2}-(3n-6)\gamma)u}.
\]
By our assumption on $\gamma$, we know that $\frac{\alpha_X}{2} - (3n-6)\gamma>0$. By applying Lemma \ref{lem=technical1} to the representation $g \mapsto \pi((g^t)^{-1})$ we obtain
\[
  |\varphi(D(v,0,-v)) - \varphi(D(v,\delta,-v-\delta))| \leq (2n-4) \widetilde{C}_X L^2 e^{-(\frac{\alpha_X}{2}-(3n-6)\gamma)v}
\]
if $|\delta|\leq\frac{v}{2}$. Setting $v=u+\delta$, these two inequalities imply that for all $\delta$ satisfying $0<\delta\leq\frac{u}{2}$,
\[
  |\varphi(D(u,0,-u)) - \varphi(D(u+\delta,0,-u-\delta))| \leq (4n-8) \widetilde{C}_X L^2 e^{-(\frac{\alpha_X}{2}-(3n-6)\gamma)u}.
\]
Suppose $u,v > 0$. Taking the supremum over all $K$-invariant unit vectors $\xi$ and $\eta$, we obtain
\[
  \| \pi(m_u) - \pi(m_v)\|_{B(X)} \leq (4n-8) \widetilde{C}_X L^2  e^{-(\frac{\alpha_X}{2} - (3n-6)\gamma)u}
\]
if $0<u<v\leq\frac{3}{2}u$. This implies that the $(\pi(m_u))_{u \in \R^+}$ satisfies the Cauchy criterion. Hence, it has a limit $P$ satisfying (\ref{item=P_limit}).
\end{proof}
We now turn to the remaining steps of the proof. The key point for \underline{Step 2} is the following lemma.
\begin{lem} \label{lem=technical_lemma}
For $K$-invariant unit vectors $\xi \in X$ and $\eta \in X^*$, let $\varphi(g) = \langle \pi(g) \xi,\eta \rangle$ as above. Again, let $D=\mathrm{diag}(e^{v_1},\ldots,e^{v_{3n-6}}) \in \mathrm{SL}(3n-6,\R)$ be a diagonal matrix with $v_1 \geq \ldots \geq v_{3n-6}$ such that $v_{2n-3} = \ldots = v_{3n-6}=t$ for some $t<0$ and $v_{n-2}>0$. Let $s = \frac{1}{2n-4} \sum_{i=1}^{3n-6} |v_i|$. Then
\begin{equation} \label{eq=general_coefficent_to_nice}
  |\varphi(D) - \varphi(D(s,0,-s))| \leq 6(n-2) \widetilde{C}_X L^2 e^{-(\alpha_X-(3n-6)\gamma)|t|}.
\end{equation}
\end{lem}
\begin{proof} 
Set $k=n-2$, and let $j$ be an index such that $v_j \geq 0$ and $v_{j+1}\leq 0$. Then the assumptions on the $v_i$'s imply that $n-2 \leq j \leq 2n-4$ (i.e., $k \leq j \leq 2k$) and 
\begin{equation} \label{eq=relation_s_vi}
  \sum_{i=1}^j v_i = - \sum_{i=j+1}^{3k} v_i = \frac{1}{2} \sum_{i=1}^{3k} |v_i| = ks.
\end{equation}
Put $D_j := D$. For each $i$ satisfying $1 \leq i < j$, we construct a diagonal matrix $D_i$ as follows. Consider first the auxiliary diagonal matrix $\widetilde{D}_i$ having the same eigenvalues as $D_{i+1}$ except for its $(i+1)^{\textrm{th}}$ one, which equals the $(i+1)^{\textrm{th}}$ eigenvalue of $D(s,0,-s)$, and its first eigenvalue, which is modified in such a way that $\widetilde{D}_i$ has determinant $1$. If $i \geq k$, we set $D_i := \widetilde{D}_i$. If $i<k$, the first eigenvalue of $\widetilde{D}_i$ may not be the largest. However, we can rearrange the $i$ first eigenvalues in decreasing order, and we define $D_i$ to be this rearranged matrix.

By construction, the $(i+1)^{\textrm{th}}$ until the $j^{\textrm{th}}$ eigenvalue of $D_i$ coincide with those of $D(s,0,-s)$, whereas the last $3k-j$ ones coincide with those of $D$. Since $D_1$ has determinant $1$, the first eigenvalue $\lambda$ of $D_1$ satisfies $1 = \lambda e^{(k-1)s} \prod_{i={j+1}}^{3k} e^{v_i} = \lambda e^{-s}$ by \eqref{eq=relation_s_vi}, so that $\lambda = e^s$. Similarly, by \eqref{eq=relation_s_vi}, we see that for $i<k$, the largest eigenvalue of $D_i$ is greater than or equal to $e^s$.

For each $i$ satisfying $1 \leq i < j$, we apply Lemma \ref{lem=coeff_GLn} and obtain
\[
  |\varphi(D_{i+1})-\varphi(\widetilde{D}_i)| \leq 2\widetilde{C}_X L^2 e^{-\alpha_X(\mu-t)-(3n-6)\gamma t},
\]
where $e^\mu$ is the minimum of the first and the $(i+1)^{\textrm{th}}$ eigenvalues of $\widetilde{D}_i$ and of $D_{i+1}$. Using that $e^\mu \geq 1$ and that $\varphi(\widetilde{D}_i)=\varphi(D_i)$, we obtain that $|\varphi(D_{i+1})-\varphi(D_i)| \leq 2\widetilde{C}_X L^2 e^{-(\alpha_X - (3n-6)\gamma)|t|}$. Summing over $i$, we obtain
\[
  |\varphi(D) - \varphi(D_1)| \leq 2j\widetilde{C}_X L^2 e^{-(\alpha_X - (3n-6)\gamma)|t|}.
\]
Conjugating by the Cartan automorphism $\theta\colon g \mapsto (g^t)^{-1}$, the above procedure and inequality (for the representation $\pi \circ \theta$ and for $D$ replaced by $\theta(D_1)$, which satisfies the assumption of the lemma for $t=-s$) yields
\[
  |\varphi(D_1) - \varphi(D(s,0,-s))| \leq 2(3k-j)\widetilde{C}_X L^2 e^{-\alpha_X s+(3n-6)\gamma s}.
\]
By \eqref{eq=relation_s_vi}, we have $|t| \leq s$, and hence
\[
  |\varphi(D) - \varphi(D(s,0,-s))| \leq 6k\widetilde{C}_X L^2 e^{-(\alpha_X-(3n-6)\gamma)|t|}.
\]
\end{proof}
Let $\mathcal{D}$ denote the set of all matrices satisfying the assumptions of Lemma \ref{lem=technical_lemma}, i.e.~matrices of the form $\mathrm{diag}(e^{v_1},\ldots,e^{v_{3n-6}}) \in \mathrm{SL}(3n-6,\R)$ with $v_1 \geq \ldots \geq v_{3n-6}$ such that $v_{2n-3} = \ldots = v_{3n-6}=t$ for some $t<0$ and $v_{n-2}>0$. A direct consequence of \underline{Step 1} and Lemma \ref{lem=technical_lemma} is the following:
\begin{equation}\label{eq=convergence_in_D}
  \lim_{D \in \mathcal D,\;\ell(D) \to \infty} \| \pi(\widetilde{m}_D) - P\|_{B(X)} = 0.
\end{equation}
This we can use to give a proof of \underline{Step 2}.
\begin{proof}[Proof of \underline{Step 2}.]
Firstly, note that we can rephrase \eqref{eq=convergence_in_D} to
\[
  \lim_{g \in K \mathcal D K,\;\ell(g) \to \infty} \|\pi(\widetilde m_g)-P\| = 0.
\]
Consider the sequence $(D_i)_{i \in \N}$ in $\mathcal{D}$ given by $D_i=D(2i,-i,-i)$. We claim that for all $i,j \in \N$ and $k \in K$, we have $D_i k D_j \in K \mathcal D K$. This claim implies that for all $i \in \N$, we have 
\begin{equation}\label{eq=piDP=P}
  \pi(\widetilde m_{D_i}) P   = \lim_{j \to \infty} \int_K \pi(\widetilde m_{D_i k D_j}) dk = P.
\end{equation}
From letting $i$ tend to infinity, we can conclude that $P^2=P$. This proves that $P$ is a projection.

It remains to prove the claim. To this end, we prove that $g :=  D_i k D_j k^{-1} \in K \mathcal D K$. Equivalently, we have to show that the $n-2$ smallest singular values of $g$ are equal and that the $(n-2)$-th largest singular value is greater than $1$. Without loss of generality, we can assume that $i \leq j$. Note that $D_i$ acts as $e^{-i}\mathrm{Id}$ on a subspace of dimension $2n-4$ of $\R^{3n-6}$ and that $(k D_jk^{-1})$ acts as $e^{-j}\mathrm{Id}$ on another subspace of dimension $2n-4$. Hence, $g$ acts as $e^{-(i+j)}$ on the intersection of these subspaces, which has dimension at least $n-2$. Moreover, since $\|g^{-1}\| \leq \| D_i^{-1} \| \| D_j^{-1} \| \leq  e^{i+j}$, we see that $e^{-(i+j)}$ is actually a singular value of $g$ with multiplicity at least $n-2$, and that it is the smallest singular value. It remains to show that the $(n-2)$-th largest singular value of $g$ is greater than $1$, or equivalently, to find a subspace of dimension $n-2$ on which $\|gx\| > \|x\|$. This subspace is, in fact, the image of $\mathrm{span}\{e_1,\dots,e_{n-2}\}$ under $k$. Indeed, for an $x$ in this subspace, we have $\| k D_j k^{-1} x \| = e^{2j} \|x\|$. Therefore, $\|gx \| \geq \|D_i^{-1}\|^{-1} \| k D_j k^{-1} x \| \geq e^{2j-i} \|x\| > \|x\|$. This proves the claim.
\end{proof}
The key to \underline{Step 3} is the following lemma.
\begin{lem}\label{lem=vanishing_noninvariant_coeff} Let $\xi$ be a $K$-invariant unit vector of $X$, let $V$ be a nontrivial irreducible representation of $K$, and let $q \colon X \to V$ be a norm-$1$ $K$-equivariant linear map. Then
\[
  \lim_{D \in \mathcal D, \ell(D) \to \infty} \|q(\pi(D) \xi)\|_V = 0.
\]
\end{lem}
\begin{proof}[Proof of Lemma \ref{lem=vanishing_noninvariant_coeff}]
Take $\xi$, $V$ and $q$ as in the statement of the lemma. We will apply \eqref{eq=coeff_noninv_SOn} to the restriction of $\pi$ to various subgroups isomorphic to $\mathrm{SO}(n)$ in $K=\mathrm{SO}(3n-6)$ and to the restriction to these subgroups of the map $K \ni k \mapsto q(\pi(\mathcal{A} k \mathcal{A}) \xi)$ for suitable choices of diagonal matrices $\mathcal{A}$. Suitable means that $\mathcal{A}$ commutes with the copy of $\mathrm{SO}(n-1)$ in the subgroup isomorphic to $\mathrm{SO}(n)$ in $K$ under consideration. First the same proof as for the $K$-bi-invariant coefficients (\underline{Step 1} and Lemma \ref{lem=technical_lemma}) shows that $\lim_{D \in \mathcal D, \ell(D) \to \infty} \|q(\pi(D)\xi)\|_V$ exists, since the function $g \mapsto \|q(\pi(g)\xi)\|_V$ is $K$-bi-invariant and, by \eqref{eq=coeff_norm_noninv_SOn}, satisfies the same local H\"older continuity estimates as the $K$-bi-invariant coefficients of $\pi$. To prove that this limit equals zero, we therefore only have to prove that
\[
  \lim_{t \to \infty} \| q(\pi(D(t,0,-t))\xi)\|_V = 0.
\]
For $i,j \in \{1,2,\dots,3n-6\}$ satisfying $i \neq j$, let us introduce the subgroup $\mathrm{SO}_{i,j}$ of $\mathrm{SO}(3n-6)$ that acts as the identity on $\mathrm{span}\{e_k \mid k \notin \{i,j\}\}$, where $e_k$ denotes the $k^{\textrm{th}}$ coordinate basis vector of $\R^{3n-6}$. Thus, for $i \neq j$, the group $\mathrm{SO}_{i,j}$ is isomorphic to $\mathrm{SO}(2)$. Consider $i\leq n-2 < j \leq 2n-4$. Let $f \colon \R^n \to \R^{3n-6}$ be an isometry sending the first coordinate vector of $\R^n$ to $e_i$, the second to $e_j$ and the orthogonal complement of the first two coordinate vectors to $\mathrm{span}\{e_k \mid 2n-4<k\leq 3n-6 \}$. The map $f$ allows us to identify $\mathrm{SO}(n)$ with the subgroup of the elements of $K$ that are the identity on the orthogonal complement of the image of $f$, and through this identification $\mathrm{SO}(n-1)$ consists of the elements that are the identity on the orthogonal complement of $\{e_j\} \cup \{e_k \mid k > 2n-4\}$. Consider the diagonal matrix $\mathcal{A}$ that has the same entries as $D(\frac{t}{2},0,-\frac{t}{2})$  except for the $i^{\textrm{th}}$ eigenvalue, which is equal to $e^{t}$, and the $j^{\textrm{th}}$ one, which is equal to $e^{-\frac{t}{2}}$. By construction, $\mathcal{A}$ commutes with $\mathrm{SO}(n-1)$ and $\ell(\mathcal{A}) = t$. Consider, as in the proof of Lemma \ref{lem=coeff_GLn}, $k_\delta \in K$ a rotation of angle $\arccos(\delta) \in [0,\pi]$ in the plane generated by the $i$-th and $j$-th coordinate vectors of $\R^2$. Then \eqref{eq=coeff_noninv_SOn} implies that
\[
  \|q(\pi(\mathcal{A} k_\delta \mathcal{A}) \xi ) - q(\pi(\mathcal{A} k_0 \mathcal{A}) \xi )\|_V \leq C(X,V)|\delta|^\alpha L^2 e^{2\gamma t}.
\]
Then $\mathcal{A} k_\delta \mathcal{A}$ coincides with $D(t,0,-t)$ on $\mathrm{span}\{e_k \mid k \notin \{i,j\}\}$, and in the plane $\mathrm{span}\{e_i,e_j\}$ its matrix is given by
\[
  \begin{pmatrix} e^t & 0 \\ 0 & e^{-t/2} \end{pmatrix} \begin{pmatrix} \delta & -\sqrt{1-\delta^2} \\ \sqrt{1-\delta^2} & \delta \end{pmatrix} \begin{pmatrix} e^t & 0 \\ 0 & e^{-t/2} \end{pmatrix} \in \mathrm{SO}(2) \begin{pmatrix} e^{x_\delta} & 0 \\ 0 & e^{y_\delta} \end{pmatrix} \mathrm{SO}(2)
\]
for some $x_\delta,y_\delta \in \R$ satisfying $x_\delta \geq y_\delta$ and ${x_\delta}+{y_\delta}=t$. If $D_\delta$ is the diagonal matrix with the same diagonal entries as $D(t,0,-t)$ except for the $i^{\textrm{th}}$ one, which is $e^{x_\delta}$, and the $j^{\textrm{th}}$ one, which is $e^{y_\delta}$, then in particular $\mathcal{A} k_\delta \mathcal{A} \in \mathrm{SO}_{i,j} D_\delta \mathrm{SO}_{i,j}$. As in the proof of \cite[Lemma 4.3]{delaatdelasalle2}, there exists a $\delta$ satisfying $|\delta|\leq e^{-t}$ such that $x_\delta= t$ and $y_\delta = 0$, implying that $D_\delta = D(t,0,-t)$. For $\delta=0$, we have $x_\delta = y_\delta = \frac{t}{2}$, and, in particular, $D_0$ commutes with $\mathrm{SO}_{i,j}$. Hence $q(\pi(\mathcal{A} k_0 \mathcal{A}) \xi) = q(\pi(D_0)\xi)$ is $\mathrm{SO}_{i,j}$-invariant, and the previous inequality becomes 
\[
  \|q(\pi(D(t,0,-t) \xi)) - q(\pi(D_0) \xi)\|_V \leq C(X,V) C^2 e^{-(\alpha - 2\gamma) t},
\]
which goes to zero as $t\to \infty$ by our assumption on $\gamma$. Hence, $q(\pi(D(t,0,-t) \xi))$ is at distance $o(1)$ from the subspace  $V_{i,j} \subset V$  of $\mathrm{SO}_{i,j}$-invariant vectors. This holds for $i\leq n-2 < j \leq 2n-4$. Similarly, by conjugating by the Cartan automorphism $g \mapsto (g^t)^{-1}$, we see that $q(\pi(D(t,0,-t) \xi))$ is at distance $o(1)$ from the subspace of $\mathrm{SO}_{i,j}$-invariant vectors for all $n-2<j\leq 2n-4 < i$. Hence, $q(\pi(D(t,0,-t) \xi))$ is at distance $o(1)$ from the intersection $\cap_{i,j} V_{i,j}$ of these subspaces, which is the subspace of vectors in $V$ invariant under all the subgroups $\mathrm{SO}_{i,j}$ for $i\leq n-2 < j \leq 2n-4$ or $n-2<j\leq 2n-4 < i$. But these subgroups generate the whole group $K$ and $V$ is assumed to have no nonzero $K$-invariant vectors. This proves that $q(\pi(D(t,0,-t) \xi))$ is at distance $o(1)$ from $0$.
\end{proof}
We can now conclude the proof of Proposition \ref{prp=strongtsln} by giving the proof of \underline{Step 3}.
\begin{proof}[Proof of \underline{Step 3}.]
As in the proof of \underline{Step 2}, we set $D_i = D(2i,-i,-i)$, and we will use that $D_i K D_j \subset \mathcal D$ for all $i,j$.

To see that Lemma \ref{lem=vanishing_noninvariant_coeff} indeed implies that every vector $\xi$ in the image of $P$ is $\pi(G)$-invariant, observe firstly that it is clear that $\xi$ is $\pi(K)$-invariant, because each measure $m_u$ is left $K$-invariant. Writing
\[
  \pi(D_1) \xi = \lim_{j \to \infty} \pi(D_1) \pi(\widetilde m_{D_j}) \xi = \lim_j \int_K \pi(D_1 k D_j)\xi dk
\]
and recalling that $D_1 k D_j \in K \mathcal D K$ for all $k \in K$, we see that $q(\pi(D_1) \xi) = 0$ for every $V$ and $q\colon X \to V$ as in Lemma \ref{lem=vanishing_noninvariant_coeff}. By the Peter-Weyl theorem for representations of compact groups on Banach spaces (see \cite[Theorem 2.5]{delasalle1}), this implies that $\pi(D_1) \xi$ is $\pi(K)$-invariant. From \eqref{eq=piDP=P}, we obtain
\[
  \xi = \pi(\widetilde m_{D_1}) P \xi= \int_K \pi(k) \pi(D_1) \xi dk,
\]
and it follows that $\xi = \pi(D_1)\xi$. Hence, $\xi$ is invariant under $\pi(g)$ for all $g$ in the closed subgroup of $G$ generated by $D_1$ and $K$, which is exactly the group $G$.
\end{proof}

\section{Property (FF$_X$) for $\mathrm{SL}(N,\mathbb{R})$ and $\mathrm{SL}(N,\mathbb{Z})$} \label{sec=ffpp}
In this section, we present a proof of Theorem \ref{thm=ffpp}, without appealing to strong property (T). Instead, we will deduce Theorem \ref{thm=ffpp} from property (T$_X$) for $\mathrm{SL}(N,\mathbb{R})$, which has been obtained in \cite{delaatdelasalle2} for reflexive Banach spaces by a much simpler argument than that in the proof of  Theorem~\ref{thm=strongt}. For non-reflexive Banach spaces, property (T$_X$) for $\mathrm{SL}(N,\mathbb{R})$ is a particular case of Theorem~\ref{thm=strongt} for isometric representations, and (although all steps are needed) obtaining just property (T$_X$) for $\mathrm{SL}(N,\mathbb{R})$ along the lines of the proof of Theorem~\ref{thm=strongt} involves simpler computations, as one does not have to take care of the growth of the representations. From this, we will be able to deduce the boundedness property of (quasi-)$1$-cocycles. In \cite{baderfurmangelandermonod}, an argument of this sort was provided for simple algebraic groups. However, there a Howe--Moore type property for isometric representations on superreflexive Banach spaces due to Shalom (see Appendix 9 in \cite{baderfurmangelandermonod}) is used, and it is not known how to extend their arguments beyond the reflexive case. In this article, we utilize a different method, namely, ``increasing the rank of the group". This part of the argument is established in Proposition \ref{prp=fromttoffpp} and is based on previous work of the second-named author \cite{mimura}. We will first state a lemma, which we strictly speaking do not need in order to prove the proposition. The reason for stating this lemma is the key role it plays if $X$ is reflexive, and because it motivates us to consider the seminorm $\mathcal N$ for a general $X$ in the proof of Proposition~\ref{prp=fromttoffpp}.
\begin{lem}\label{lem=dual}
Let $W$ be a topological group, and let $(Y,\pi)$ be a weakly almost periodic $W$-space, i.e.~a Banach space $Y$ with an isometric representation $\pi \colon W \rightarrow O(X)$ with relatively weakly compact orbits (i.e.~the weak closure is weakly compact). Then for any $\eta \in Y$, we have $\sup_{w\in W} \|\pi(w)\eta -\eta \| \geq  \|\overline{\eta} \|$, where $\xi\mapsto \overline{\xi}$ denotes the quotient map $Y\twoheadrightarrow Y/Y^{\pi(W)}$.
\end{lem}
\begin{proof}
Let $\eta \in Y$, and let $B$ denote the closure of the convex hull of the set $\{\pi(w)\eta \mid w\in W\}$. Then $B$ is weakly compact because the convex hull of a relatively weakly compact set is relatively weakly compact. By the Ryll-Nardzewski fixed point theorem, there exists a $\pi(W)$-fixed point in $B$. If $\sup_{w\in W} \|\pi(w)\eta -\eta \| <  \|\overline{\eta} \|$, then $B\cap Y^{\pi(W)}=\emptyset$, but this is absurd.
\end{proof}
The main step in the proof of Theorem~\ref{thm=ffpp} is to prove the following proposition. We state it in a full generality. Recall that for an associative and unital topological ring (not necessarily commutative) $A$ and for $n\geq 2$, the matrix $e_{i,j}(a)$, where $a\in A$, $1\leq i\leq n$ and $1\leq j\leq n$ with $i\ne j$, denotes the element in $\mathrm{Mat}(n,A)$ whose diagonal entries are all $1$ and whose off-diagonal entries are $0$ for all but the $(i,j)$-th one, which equals $a$. We call such a matrix an elementary matrix. The elementary group $\mathrm{E}(n,A)$ denotes the multiplicative group generated by all elementary matrices in $\mathrm{Mat}(n,A)$ equipped with the natural topology induced from the topology on $A$. Note that by the commutation relation $[e_{i,j}(a_1),e_{j,k}(a_2)]=e_{i,k}(a_1a_2)$ for $i\ne j\ne k\ne i$ and $a_1,a_2\in A$, the group $\mathrm{E}(n,A)$ is compactly generated if the ring $A$ is compactly generated and $n\geq 3$. If $A$ is a commutative and euclidean ring (such as $\mathbb{R}$ or $\mathbb{Z}$), then $\mathrm{E}(n,A)$ coincides with $\mathrm{SL}(n,A)$. 
\begin{prp} \label{prp=fromttoffpp}
Let $A$ be an associative and unital topological ring that is compactly generated, let $X$ be a Banach space and $n \geq 2$. Assume that the pair $\mathrm{E}(n,A)\ltimes A^n > A^n$ has weak relative property (T$_X$). Then the pair $\mathrm{E}(n+1,A)\ltimes A^{n+1} > A^{n+1}$ has relative property (FF$_X$).
\end{prp}
We refer to Definition~\ref{def=ffpp} for the notion of weak relative property (T$_X$) and relative property (FF$_X$). Proposition \ref{prp=fromttoffpp} is a generalization of \cite[Theorem~6.4]{mimura} beyond the superreflexive case. The essential idea of the first part of the proof comes from \cite[Section~3]{mimura}, but we include a full proof for completeness. The proof is formulated as quantitatively as possible in terms of the norm bound of quasi-$1$-cocycles.
\begin{proof}
Let $Q$ be a symmetric (i.e.~closed under taking the additive inverse) compact generating set of $A$ that contains the identity $1$. Let $G=\mathrm{E}(n+1,A)\ltimes A^{n+1}$, and realize $G$ inside $\mathrm{E}(n+2,A)$ by putting $\mathrm{E}(n+1,A)$ in the upper left corner of $\mathrm{E}(n+2,A)$ and $H:=A^{n+1}$ in the first $n+1$ entries of the $(n+2)$-th column. Then the set $S:=\{e_{i,j}(q) \in \mathrm{E}(n+2,A) \mid i\ne j, q\in Q\} \cap G$ is a compact generating set of $G$. We define $H_1 ,H_2 < H$ in the following way: $H_1 \cong A$ is the additive group inside $H$ on the $(1,n+2)$-th entry of $\mathrm{E}(n+2,A)$, and $H_2 \cong A^n$ is the additive group inside $H$ corresponding to the ``rest of $H$'', i.e.~the $(2,n+2)$-th, $(3,n+2)$-th, $\ldots$ and $(n+1,n+2)$-th entries of $\mathrm{E}(n+2,A)$. We define $S_0\subset S$ as the set of all elements of $S$ whose $(i,1)$-th entry is zero for $i=2,\dots,n+1$. Finally, define $G_1,G_2 <G$ and $L \triangleleft G_1$ by
\begin{align*}
&G_1:=\left\{\left(\begin{array}{ccc}
1 & (v')^t & 0 \\
0 & \gamma' & 0 \\
0 & 0 &1
\end{array}\right)\middle| \ \gamma'\in \mathrm{E}(n,A), v'\in A^n
\right\},\\ 
&G_2:=\left\{\left(\begin{array}{ccc}
1 & 0 & 0 \\
0 & \gamma' & v' \\
0 & 0 &1
\end{array}\right)\middle| \ \gamma'\in \mathrm{E}(n,A), v'\in A^n
\right\}, \ \mathrm{and}\ \\
&L:=\left\{g \in G_1 \mid \gamma'=I_n\right\}.
\end{align*}
Let $\rho \colon G\to O(X)$ be an isometric representation. Because $H$ is normal in $G$, the representation $\rho$ can be restricted to $X^{\rho(H)}$ as an isometric representation of $G$. Let $c\colon G\to X$ be a quasi-$1$-cocycle into $\rho$. Then $c$ induces a map $\overline{c}\colon G\to X/X^{\rho(H)}$, and $\overline{c}$ is a quasi-$1$-cocycle into $\overline{\rho}$. Define $\Delta<\infty $ as the defect of $c$ and $M<\infty$ as $\sup_{s\in S}\|c(s)\|$. 

For $x \in X$ let us denote $\mathcal{N}(x) = \sup_{w \in H} \|\rho(w) x - x\|$. Then $\mathcal{N}$ is a $\rho(G)$-invariant seminorm on $X$ (recall that $H$ is normal in $G$), which plays the role of $\|\cdot\|_{X/X^{\rho(H)}}$. In fact, if the representation $(X,\rho)$ is weakly almost periodic (see \cite{baderrosendalsauer} for the definition), which is in particular the case if $X$ is reflexive, then $\|x\|_{X/X^{\rho(H)}} \leq \mathcal{N}(x) \leq  2\|x\|_{X/X^{\rho(H)}}$ (see the discussion following Definition~\ref{def=ffpp}).

The proof roughly consists of three parts. Firstly, we will show that \linebreak $\sup_{h \in H} \mathcal{N}(c(h)) < \infty$. One of the crucial points is that $H_1$ commutes with $S_0$. Therefore, for any $h\in H_1$ and $s\in S_0$, we have that
\begin{align*}
\|\rho(s)c(h)-c(h)\| &\leq \|c(sh) -c(s)-c(h)\|+\Delta \\
&\leq \|c(hs)-c(h)\| +M+\Delta \\
& \leq \| \rho(h)c(s)\|+M+2\Delta \leq 2(M+\Delta).
\end{align*}
Another point is that the pairs $G_1> L$ and $G_2> H_2$ are isomorphic copies of the pair $\mathrm{E}(n,A)\ltimes A^n > A^n$ (note that $\gamma'$ should be changed to $(\gamma'^{-1})^{t}$ to construct a concrete group isomorphism for the first pair, but that $\mathrm{E}(n,A)$ and $\mathrm{E}(n,A) \cap S$ are, as sets, stable under taking this operation). Therefore, by assumption, the minimum $\tilde{\kappa}$ of $\tilde{\kappa}(G_1,L, S_0\cap G_1,\rho\mid_{G_1})$ and $\tilde{\kappa}(G_2,H_2, S_0\cap G_2,\rho\mid_{G_2})$ is strictly positive (see Definition~~\ref{def=ffpp} for these constants). Hence, $\sup_{w \in L} \| \rho(w) c(h)-c(h)\| \leq 2\tilde{\kappa}^{-1} (M+\Delta)$ and $\sup_{w \in H_2} \| \rho(w) c(h)-c(h)\| \leq 2\tilde{\kappa}^{-1} (M+\Delta)$. The third point is that $H\subset LH_2LH_2$. Hence, for any $h\in H_1$, we have that
\[
\sup_{w \in H} \| \rho(w) c(h)-c(h)\|\leq 8 \tilde{\kappa}^{-1}(M+\Delta),
\]
or equivalently, $\mathcal{N}(c(h)) \leq 8 \tilde{\kappa}(M+\Delta)$ for all $h \in H_1$. We can apply the same reasoning with $H_1$ replaced by the additive group sitting at the $(i,n+2)$-th entry for $1 \leq i \leq n+1$. At the end we get
\[
\sup_{c \in H} \mathcal{N}(c(h)) \leq (n+1) \tilde{\kappa}(8M+9\Delta),
\]
which verifies our assertion above.

Secondly we claim that, in fact, $\mathcal{N}(c(\cdot))$ is bounded on the whole group $G$. A key step in the proof of this claim is the identity $(\gamma,0)(I_{n+1},u)=(I_{n+1},\gamma u) (\gamma,0)$, where we write $g=(\gamma,u)$ where $\gamma\in \mathrm{E}(n+1,A)$ and $u\in A^{n+1}$. We identify $v \in A^{n+1}$ with $(I_{n+1},v) \in H$. Then, from the identity above and the quasi-$1$-cocycle relation, for any $\gamma \in \mathrm{E}(n+1,A)$ and $u\in A^{n+1}$, we have that
\begin{equation}\label{eq:cocycle_conjugation}
\|c(\gamma u) -\rho((\gamma,0))c(u) -(I- \rho(\gamma u))c((\gamma,0))\|\leq 2\Delta. 
\end{equation}
By passing to $\mathcal{N}$, using that $\mathcal{N}(x) \leq 2\|x\|$ and that $\mathcal{N}$ is $\rho(G)$-invariant, we obtain that 
\begin{align*} \mathcal{N}(\rho(\gamma u)c((\gamma,0))-c((\gamma,0)) ) &\leq 4\Delta + \mathcal{N}(c(\gamma u)) + \mathcal{N}(c(u))\\ 
& \leq 2(n+1)\tilde{\kappa}^{-1}(8M+11\Delta)
\end{align*}
by the boundedness of $\mathcal{N}(c(\cdot))$ on $H$, as we have showed in the first step. Denote $M'= 2(n+1)\tilde{\kappa}^{-1}(8M+11\Delta)$. Fix $\gamma \in \mathrm{E}(n+1,A)$ and move $u\in H$ arbitrarily. Then
\[
\sup_{w\in H} \mathcal{N}(\rho(w) c((\gamma,0))-c((\gamma,0)) )\leq M'.
\]
Now fix $w \in H$ and denote $x_n = \frac{1}{n} \sum_{i=0}^{n-1} \rho(iw) c((\gamma,0))$, where $iw = w+\dots+w$ ($i$ times). By using that $\mathcal{N}$ is a seminorm, we have $\mathcal{N}(x_n - c((\gamma,0))) \leq M'$. In particular, $\|(I-\rho(w)) (x_n - c((\gamma,0)))\| \leq M'$. This implies that
\begin{align*}
\|(I-\rho(w)) c((\gamma,0))\| &\leq M' + \|(I-\rho(w)) x_n\| \\
&= M' + \frac 1 n \|c((\gamma,0)) - \rho(nw) c((\gamma,0))\|\\
&\leq M' +\frac{2}{n}\|c((\gamma,0))\|
\end{align*}
By taking $n \to \infty$, we obtain that $\|(I-\rho(w)) c((\gamma,0))\| \leq M'$ for any $w\in H$, and hence that $\mathcal{N}(c((\gamma,0))) \leq M'$, as requested. From this, it is easy to see that $\mathcal{N}(c(\cdot))$ is bounded on the whole $G$, as we claimed. In what follows, we only need the boundedness on $\mathrm{E}(n+1,A)$.

Finally, we will prove that $c$ itself is bounded on $H$. By \eqref{eq:cocycle_conjugation} and what we just proved in the second step, we conclude that for any $\gamma \in \mathrm{E}(n+1,A)$ and any $u\in H$,
\[ \|c(\gamma u) \| \leq M' + 2\Delta + \|c(u)\|.\]
The final observation is that every element $v\in H\simeq A^{n+1}$ can be written as the sum of (at most) $n+1$ elements of the form $\gamma u$, where $\gamma \in \mathrm{E}(n+1,A)$ and $u \in S\cap H$. This leads us to the inequality
\[
\sup_{v\in H}\| c(v)\| \leq (n+1)(M'+M+3\Delta)\leq (n+1)^2\tilde{\kappa}^{-1}(17M+25\Delta),
\]
which completes our proof.
\end{proof}
We also state the following easy observation for the proof of Theorem~\ref{thm=ffpp}.
\begin{lem}\label{lem=easy}
Let $A$ be an associative and unital topological ring that is compactly generated, let $X$ be a Banach space and $n \geq 2$. Assume that $\mathrm{E}(n,A)$ has property (T$_X$). Then the pair $\mathrm{E}(n,A)\ltimes A^n > A^n$ has weak relative property (T$_X$).
\end{lem}
\begin{proof}
In a similar way as in the proof of Proposition \ref{prp=fromttoffpp}, take a natural embedding of $\mathrm{E}(n,A)\ltimes A^n$ into $\mathrm{E}(n+1,A)$ that sends $\mathrm{E}(n,A)$ to the upper left corner of $\mathrm{E}(n+1,A)$ and $A^n$ to the first $n$ entries of the $(n+1)$-th column. We identify $\mathrm{E}(n,A)$ and $A^n$ with their respective images under this embedding. Let $Q$ be a compact generating set of $A$ that contains $1$. Let $S$ be defined as \linebreak $S:=\{e_{i,j}(q) \in \mathrm{E}(n+1,A) \mid i \ne j,q\in Q\}\cap \left(\mathrm{E}(n,A)\ltimes A^n \right)$. In what follows, we will show that for any isometric representation $\rho$ of $\mathrm{E}(n,A)\ltimes A^n$ on $X$, 
\[
\tilde{\kappa}_2 \geq  \frac{1}{n(2\tilde{\kappa}_1^{-1}+1)},
\]
where
\begin{align*}
 \tilde{\kappa}_1 &:=\tilde{\kappa}(\mathrm{E}(n,A), \mathrm{E}(n,A), S\cap \mathrm{E}(n,A),\rho\mid_{\mathrm{E}(n,A)}), \textrm{ and}\\
 \tilde{\kappa}_2 &:=\tilde{\kappa}(\mathrm{E}(n,A)\ltimes A^n, A^n, S,\rho).
\end{align*}
This assertion will immediately prove the lemma.

Fix such a $\rho$ and take an arbitrary $\xi \in X$. We define $b\colon \mathrm{E}(n,A)\ltimes A^n\to X$ by $b(g):=\rho(g)\xi -\xi$. This defines a $1$-cocycle into $\rho$ (in fact, a $1$-coboundary into $\rho$), and hence $b$ satisfies the cocycle relation $b(g_1g_2)=b(g_1)+\rho(g_1)b(g_2)$ for all $g_1,g_2 \in \mathrm{E}(n,A)$. In particular, by the triangle inequality, for every $g_1,g_2 \in \mathrm{E}(n,A)\ltimes A^n$, we have $\|b(g_1g_2)\| \leq \|b(g_1)\| +\|b(g_2)\|$. 

By the definition of $\tilde{\kappa}_1$, we have that 
\[
\tilde{\kappa}_1 \sup_{\gamma \in \mathrm{E}(n,A)} \|b(\gamma)\| \leq  \sup_{s\in S\cap \mathrm{E}(n,A)} \|b(s)\| \leq \sup_{s\in S} \|b(s)\|.
\]
Note that $\tilde{\kappa}_1>0$ by assumption. Then, by the triangle inequality, 
\[
\sup_{ h'\in \mathrm{E}(n,A)\cdot S \cdot \mathrm{E}(n,A)} \|b(h')\| \leq  (2\tilde{\kappa}_1^{-1}+1) \sup_{s\in S} \|b(s)\|.
\]
Here $\mathrm{E}(n,A)\cdot S \cdot \mathrm{E}(n,A)$ denotes the product subset in $\mathrm{E}(n,A)\ltimes A^n$. 

Finally, note that, as in the final observation in proof of Proposition~\ref{prp=fromttoffpp}, any $h\in A^n$ may be be written as the product of (at most) $n$ elements in the set $\mathrm{E}(n,A)\cdot S \cdot \mathrm{E}(n,A)$. Therefore, we conclude that
\[
\sup_{h \in A^n} \|b(h)\| \leq  n(2\tilde{\kappa}_1^{-1}+1) \sup_{s\in S} \|b(s)\|,
\]
which verifies our assertion above.
\end{proof}
From Lemma~\ref{lem=easy} we can in fact conclude weak relative property (T$_X$) for the pair $\mathrm{E}(n,A)\ltimes A^n > \mathrm{E}(n,A)\ltimes A^n$.
\begin{proof}[Proof of Theorem~\ref{thm=ffpp}]
Let $\beta < \frac{1}{2}$, and let $X$ be a Banach space satisfying \eqref{eq=d_kX_grows_slowly}. Instead of appealing to strong property (T), we employ property (T$_X$). As we argued after Definition~\ref{def=ffpp} it follows from Theorem \ref{thm=strongt} that for such an $X$, there exists an $N'$ such that $\mathrm{SL}(N',\mathbb{R})$ has property (T$_X$). In the case that $X$ is reflexive, the proof of this assertion is considerably less involved than the proof of the full result of Theorem~\ref{thm=strongt} and follows from the work of \cite{delaatdelasalle2} and using the Howe-Moore property from \cite{veech} (see Remark \ref{rem=remstrongt}). Even if $X$ is not reflexive, the proof is simpler, because here we only have to deal with isometric representations. 

Hence, it follows that $\mathrm{SL}(N',\mathbb{R})$ has property (T$_{X}$). Then, by Lemma~\ref{lem=easy}, the pair $\mathrm{SL}(N',\mathbb{R})\ltimes \mathbb{R}^{N'}> \mathbb{R}^{N'}$ satisfies weak relative property (T$_{X}$). 

By applying Proposition~\ref{prp=fromttoffpp} for $A=\mathbb{R}$, it follows that the pair $\mathrm{SL}(N'+1,\mathbb{R})\ltimes \mathbb{R}^{N'+1}> \mathbb{R}^{N'+1}$ has relative property (FF$_X$). This implies that for every isometric representation $\rho \colon \mathrm{SL}(N'+2,\mathbb{R}) \to O(X)$ and any quasi-1-cocycle $c\colon \mathrm{SL}(N'+2,\mathbb{R})\to X$ into $\rho$, there exists a constant $M>0$ such that
\[
\sup\left\{\|c(g)\|\ \middle| \ g \in \bigcup_{i\ne j, 1\leq i\leq N'+2, 1\leq j\leq N'+2}\{e_{i,j}(a) \mid a\in \mathbb{R}\} \right\} \leq M.
\]
In order to see this, embed $\mathrm{SL}(N'+1,\mathbb{R})\ltimes \mathbb{R}^{N'+1}$ into $\mathrm{SL}(N'+2,\mathbb{R})$ in several ways (more precisely, choose several different entries among $\{1,2,\ldots ,N'+2\}$). 

Finally, observe that $\mathrm{SL}(N'+2,\mathbb{R})$ is boundedly generated by elementary matrices in $\mathrm{Mat}(N'+2,\mathbb{R})$, i.e.~there exists an integer $k=k(N'+2)$ such that $\mathrm{SL}(N'+2,\mathbb{R})=\left(\bigcup_{i\ne j} \{e_{i,j}(a)\mid a\in \mathbb{R}\}\right)^k$. This easily follows because $\mathbb{R}$ is a field. We conclude that for $c$ as above, 
\[
\sup_{g\in \mathrm{SL(N'+2,\mathbb{R})}} \|c(g)\| \leq kM+(k-1)\Delta \quad (<\infty),
\]
where $\Delta$ is the defect of $c$. Therefore, $N:=N'+2$ satisfies our conditions.
\end{proof}
\begin{rem}\label{remark=steinberg}
The group $\mathrm{E}(n,A)$ can be viewed as the elementary Chevalley group $\mathrm{E}(\Phi,A)$ associated with the root system $\Phi$ of type $A_{n-1}$. If we consider the case where $A$ is commutative, then an assertion similar to Proposition~\ref{prp=fromttoffpp} holds for elementary Chevalley groups associated with reduced irreducible classical root systems of other types. However, if the root system is not simply-laced, the proof needs more care. From this, it is probably possible to extend the result of Theorem~\ref{thm=ffpp} to connected simple Lie groups of sufficiently large rank.

Another remark is that Proposition~\ref{prp=fromttoffpp} can be generalized to the case of (Kac--Moody--)Steinberg groups, in the sense of \cite[Subsection~6.1]{ershovjaikinzapirain}, with appropriate pairs of groups if the ring $A$ is discrete (cf.~\cite[Appendix~A]{ershovjaikinzapirain}). This version of the result is used by the second-named author in \cite{mimura15}.
\end{rem}
The proof of Theorem~\ref{thm=ffpp} implies the following corollary, which will play a key role in \cite{mimurasako}.
\begin{cor}\label{cor=ffppforSLnZ}
For every $\beta < \frac{1}{2}$, there exists an $N$ such that for any $X$ satisfying \eqref{eq=d_kX_grows_slowly}, the group $\mathrm{SL}(N,\mathbb{Z})$ has property (FF$_X$). 
\end{cor}
\begin{proof}
Property (T$_{X}$) for $\mathrm{SL}(N',\mathbb{Z})$ can be decuded from the corresponding assertion with $A=\mathbb{R}$ by the method of $L^2$-induction (see \cite{baderfurmangelandermonod}). From this, together with Lemma~\ref{lem=easy}, we obtain weak relative property (T$_X$) for the pair $\mathrm{SL}(N'+1,\mathbb{Z})\ltimes \mathbb{Z}^{N'+1}> \mathbb{Z}^{N'+1}$. Then we can apply Proposition~\ref{prp=fromttoffpp} for $A=\mathbb{Z}$ and proceed along the same lines as in the proof of Theorem \ref{thm=ffpp}. Note that in the case of $A = \mathbb{Z}$, unlike that of $A=\mathbb{R}$, the bounded generation (of $\mathrm{SL}(N,\mathbb{Z})$ by elementary matrices for $N\geq 3$) is highly non-trivial. This was proved by Carter and Keller \cite{carterkeller}, see also \cite{shalom}. 
\end{proof}
We finally point out that regarding property (FF$_{X}$), we do not know a direct argument of taking inductions to ($L^2$-integrable) lattices if $(X,\rho)$ is not the contragredient representation of an isometric representation on a dual Banach space. The problem lies in the deduction of the boundedness of the original quasi-cocycle from that of the induced quasi-cocycle. This issue becomes harmless if we treat fixed point properties instead of boundedness properties (and thus the proof of Corollary~\ref{cor=fpplattices} works without any issues). 
Burger and Monod \cite[Corollary~11]{burgermonod} resolved the problem above in the affirmative if the original isometric representation $(X,\rho)$ is separable and contragredient.

\begin{rem}\label{rem=ell1}
As announced in Remark~\ref{rem=fpffp}, we sketch the proof of the fact that $\mathrm{SL}(4,\mathbb{Z})$ has property (FF$_{\ell^1_0}$). First, we use the fact that the distance function on $\ell^1$, i.e.~the map $(\xi,\eta)\mapsto \|\xi-\eta\|_{\ell^1}$, is a conditionally negative defnite kernel (for definitions and details we refer to \cite[Appendix~C]{bekkadelaharpevalette}). It is well known that the pair $\mathrm{SL}(2,\mathbb{Z})\ltimes \mathbb{Z}^2 > \mathbb{Z}^2$ has relative property (T) (see \cite{bekkadelaharpevalette}), which is equivalent to ``relative property (FH)''. By combining these facts, we may conclude that for any affine isometric action of $\mathrm{SL}(2,\mathbb{Z})\ltimes \mathbb{Z}^2$ on any closed subspace of $\ell^1$, its $\mathbb{Z}^2$-orbits are bounded. From this, it is not difficult to see that the pair $\mathrm{SL}(2,\mathbb{Z})\ltimes \mathbb{Z}^2> \mathbb{Z}^2$ has weak relative property (T$_{\ell^1_0}$). Indeed, observe that the countable direct $\ell^1$-sums of $\ell^1_0$ is isometrically isomorphic to a closed subspace of $\ell^1$. If $\tilde{\kappa}$ in the current setting were zero for some isometric representation, then by taking an $\ell^1$-direct sum in an appropriate manner, we could construct an affine isometric action of $\mathrm{SL}(2,\mathbb{Z})\ltimes \mathbb{Z}^2$ on some closed subspace of $\ell^1$ whose $\mathbb{Z}^2$-orbits are unbounded. This is absurd, and we are done. Then, Proposition~\ref{prp=fromttoffpp} implies that the pair $\mathrm{SL}(3,\mathbb{Z})\ltimes \mathbb{Z}^3 > \mathbb{Z}^3$ has relative property (FF$_{\ell^1_0}$). The aforementioned bounded generation of Carter and Keller ends our proof.
\end{rem}

\subsection*{Acknowledgements}
We thank the referee for several valuable suggestions and remarks that 
improved the content of the article. We also thank the organizers of the 
International Conference on Banach methods in Noncommutative Geometry at 
Wuhan University, where the work that lead to this article was 
initiated.

\end{document}